\documentclass[12pt]{amsart}
\usepackage[T1]{fontenc}
\usepackage{amsmath}
\usepackage{amssymb}
\usepackage{dsfont}
\usepackage{color}
\usepackage{url}
\usepackage{hyperref}
\usepackage{mathtools}
\mathtoolsset{showonlyrefs}
\usepackage{anysize}
\usepackage{geometry}
\usepackage{ulem}
\marginsize{2.5cm}{2.5cm}{2.5cm}{2.5cm}
\theoremstyle{plain}
\numberwithin{equation}{section}
\newtheorem{theorem}{Theorem}[section]
\newtheorem{lemma}[theorem]{Lemma}

\newtheorem{definition}[theorem]{Definition}
\newtheorem{remark}[theorem]{Remark}

\newcommand{\ud}{\mathrm{d}}

\begin{document}
\title[Stability for the multivariate geometric Brownian motion]{\sc{The stability of the multivariate geometric Brownian motion as a bilinear matrix inequality problem}}

\author{Gerardo Barrera}
\address{Center for Mathematical Analysis, Geometry and Dynamical Systems (CAMGSD), Instituto Superior T\'ecnico, Universidad de Lisboa, Lisbon, Portugal.
\url{https://orcid.org/0000-0002-8012-2600}}
\thanks{*Corresponding author: Gerardo Barrera.}
\email{gerardo.barrera.vargas@tecnico.ulisboa.pt}

\author{Eyleifur Bjarkason}
\address{Science Institute and Faculty of Physical Sciences, University of Iceland. Dunhagi 5, 107 Reykjav\'{i}k, Iceland.}
\email{eyleifur@gmail.com}

\author{Sigurdur Hafstein}
\address{Science Institute and Faculty of Physical Sciences, University of Iceland. Dunhagi 5, 107 Reykjav\'{i}k, Iceland.
\url{https://orcid.org/0000-0003-0073-2765}
}
\email{shafstein@hi.is}

\subjclass[2000]{Primary 60H10, 93D05, 93D30; Secondary  34F05, 37H30, 93D23}
\keywords{Baker--Campbell--Hausdorff--Dynkin Formula; Bilinear Matrix Inequality; Exponential Stability; Linear Matrix Inequality; Linear Stochastic Differential Equations; Lyapunov Functions; Magnus expansion; Multivariate Geometric Brownian Motion; Stability}

\begin{abstract}
In this manuscript, we study the stability of the origin for
the multivariate geometric Brownian motion.
More precisely,
under suitable sufficient conditions,
we construct a Lyapunov function
such that the origin of the multivariate geometric Brownian motion is globally asymptotically stable in probability. Moreover, we show that such conditions can be rewritten as a  Bilinear Matrix Inequality (BMI) feasibility problem.
We stress that no commutativity relations between the drift matrix and the noise dispersion matrices are assumed and therefore the so-called Magnus representation of the solution of the multivariate geometric Brownian motion is complicated.
In addition, we exemplify our method in numerous specific models from the literature such as random linear oscillators, satellite dynamics, inertia systems,
diagonal and non-diagonal noise systems, cancer self-remission  and smoking.
\end{abstract}
\maketitle


\section{\textbf{Introduction}}\label{sub:MBM}\hfill

\noindent
Random dynamical systems and in particular stochastic differential equations  (for shorthand SDEs) are used in science in the modelling of  systems driven by deterministic forces and external random fluctuations, see the monographs~\cite{Arnold,Capasso,Cufaro,Karatzas,
Khasminskii,Kloeden,Mao,Oks00,Pavliotis}.

Let $n,\ell \in \mathbb{N}:=\{1,2,\ldots\}$ be fixed and let
$\mathsf{M}_n$ be the set of all $n\times n$ matrices with real entries.
For a given $x\in \mathbb{R}^n$,
let $(X(t;x))_{t\geq 0}$ be the unique strong solution of the autonomous (constant coefficients) linear $n$-dimensional SDE in the It\^o sense
\begin{equation}\label{eq:modelito}
\left\{
\begin{array}{r@{\;=\;}l}
\ud X(t;x) & AX(t;x) \ud t+\sum\limits_{j=1}^{\ell}B_jX(t;x)\ud W_j(t),\quad t\geq 0,\\
X(0;x) & x,
\end{array}
\right.
\end{equation}
where $A,B_1,\ldots,B_\ell \in \mathsf{M}_n$ and $W_j:=(W_j(t))_{t\geq 0}$ for $j\in \{1,\ldots,\ell\}$ are independent and identically distributed
standard one-dimensional Brownian motions. Since the SDE~\eqref{eq:modelito} is linear, the existence and uniqueness of the path-wise strong solution
of~\eqref{eq:modelito} is asserted, for instance, by Theorem~3.1 in~\cite{Mao} or Theorem~4.5.3
in~\cite{Kloeden}.
From now on, we denote by $(\Omega,\mathcal{F}, (\mathcal{F}_t)_{t\geq 0},\mathbb{P})$ the filtered complete probability space satisfying the usual conditions, see Definition~2.25  of~\cite{Karatzas},
where the independent standard Brownian motions $W_1,\ldots,W_\ell$ are defined, and denote by $\mathbb{E}$ the expectation with respect to the probability measure $\mathbb{P}$. Hence, $(X(t;x))_{t\geq 0}$ can be taken as a stochastic process in the same filtered
probability space $(\Omega,\mathcal{F}, (\mathcal{F}_t)_{t\geq 0},\mathbb{P})$ and such that its marginal at a fixed time $t\geq 0$, $X(t;x)$, takes values in $\mathbb{R}^n$.
It is not hard to write~\eqref{eq:modelito} in the Stratonovich sense (here denoted by $\circ$)
\begin{equation}\label{eq:modelstra}
\left\{
\begin{array}{r@{\;=\;}l}
\ud X(t;x) & \Big(A-\frac{1}{2}\sum\limits_{j=1}^{\ell}B^2_j\Big)X(t;x)+\sum\limits_{j=1}^{\ell}B_jX(t;x)\circ\ud W_j(t),
\quad t\geq 0,\\
X(0;x) & x,
\end{array}
\right.
\end{equation}
see Section~4.9 in~\cite{Kloeden}.
The solution of~\eqref{eq:modelito} is called a multivariate geometric Brownian motion.

In the sequel, we provide some literature for the  geometric Brownian motion and its applications that has been produced in recent years.
The (multivariate) geometric Brownian motion is an important class of models in mathematical finance in the modeling of stock
prices, see~\cite{Hu,Lima,Mantegna,MR05,PK2013,Pirjoluno,Pirjol,PP06,Shreve,Stojkoski}. It is of interest to compute or estimate 
expected values of observables of the process at a finite time horizon.
It has been used in physics as a model 
in turbulence theory~\cite{Giordano},
harmonic oscillators with multiplicative noise~\cite{GittermanMulti,Gitterman}, inertia systems~\cite{Hutwo},
and it also serves as an important model in biology such as cancer and tumor stability~\cite{Oroji,Sarkar}, smoking dynamics~\cite{Lahrouz,Sharomi}, etc. 
Moreover, it is used in heat propagation, cosmology and statistical field theory,
neuron models, gene expression, molecular motors, see the introduction of~\cite{Giordano}
and the references therein.
We also refer to~\cite{TrefethenEmbree,Trefethen} for models and stability of hydrodynamic.
Moreover, it is an important and well-studied mathematical
object in its own right and a prototype for multiplicative noise with the random term directly proportional to the current state of the system, see~\cite{Appleby,BHP,Bycz,Dietz,Duf,Gitterman,Haskovec} and the references therein.

For $U,V\in \mathsf{M}_n$ we denote their  Lie commutator by $[U,V]:=UV-VU$. Let $O_{n \times n}$ be the zero element of
$\mathsf{M}_n$.
Under the additional commutativity conditions
\[
[A,B_j]=O_{n \times n}\quad \textrm{ and } \quad [B_j,B_k]=O_{n \times n}\quad \textrm{ for all } \quad j,k\in \{1,\ldots,\ell\},\]
a closed expression for the solution
of~\eqref{eq:modelito} is given by
\begin{equation}\label{ec:comu}
X(t;x)=\exp\left(\Big(A-\frac{1}{2}\sum_{j=1}^{\ell}B^2_j\Big)t+\sum_{j=1}^{\ell}B_jW_j(t)\right)x\ \ \textrm{ for any }\ \ x\in \mathbb{R}^n\ \ \textrm{ and }\ \  t\geq 0,
\end{equation}
see for instance Section~3.4 in~\cite{Mao} or Section~4.8
in~\cite{Kloeden}.
Nevertheless, in general,
for $n\in \mathbb{N}\setminus\{1\}$, we point out that a closed-form expression of~\eqref{eq:modelito} is not straightforward due to the celebrated Baker--Campbell--Hausdorff--Dynkin formula for matrix exponentials, see Chapter~5 in~\cite{Hall} and~\cite{Magnus}.
We remark that solving non-commutative stochastic systems such as~\eqref{eq:modelito} is not an easy task.
Indeed, when no commutativity properties between $A,B_1,\ldots,B_\ell \in \mathsf{M}_n$ are assumed, the stochastic extension of the Baker--Campbell--Hausdorff--Dynkin formula for linear systems of SDEs yields a complicated closed-form expression of the solution to~\eqref{eq:modelito} in terms of iterated
(nested) commutators, see Theorem~1 of~\cite{Kamm} and Theorem~2.1 in~\cite{Yamato}.
In other words,
the solution of~\eqref{eq:modelito} has a Magnus expansion (exponential shape), which is given by
\begin{equation}\label{eq:shape}
X(t;x)=\exp\left(Y(t)\right)x\quad \textrm{ for any }\quad x\in \mathbb{R}^n\quad \textrm{ and }\quad t\geq 0,
\end{equation}
where the stochastic exponent (logarithm) process $(Y(t))_{t\geq 0}$ is given in Item~(2) of Theorem~1
in~\cite{Kamm}.
Such exponent $Y(t)$ of~\eqref{eq:shape} is expressed in terms of iterated commutators.
For instance, for $\ell=1$ it reads as follows
\begin{equation}
\begin{split}
Y(t)&=B_1 W_1(t)+At+[B_1,A]\left(\frac{1}{2}tW_1(t)-\int_{0}^t W_1(s)\ud s\right)-B^2_1\frac{t}{2}
\\
&\quad+[[A,B_1],B_1]\left(\frac{1}{2}\int_{0}^t W^2_1(s)\ud s-\frac{1}{2}W_1(t)\int_{0}^{t}W_1(s)\ud s+\frac{1}{12}tW^2_1(t)\right)\\
&\quad+[[A,B_1],A]\left(\int_{0}^t s W_1(s)\ud s-\frac{1}{2}t\int_{0}^{t}W_1(s)\ud s-\frac{1}{12}t^2 W_1(t)\right)+\cdots.
\end{split}
\end{equation}
The missing remainder terms in the exponent $Y(t)$ contain only terms that include
higher order iterated commutators of $A$ and $B_1$, see Section~3.1 in~\cite{Kamm}.
Roughly speaking,
commutators are essential for representing the solution of~\eqref{eq:modelito}.
Moreover, the effect of
non-commutativity
yields \textit{non-linear terms} in the stochastic exponent with respect to the time variable.
We recommend the survey~\cite{Blanes} and the references therein
for the Magnus expansions and its various applications in physics, numerical analysis and perturbation theory.
We also recommend Chapter~4 in the Ph.D.~thesis~\cite{Burragethesis} for examples of the Magnus formula in non-commutative systems,
Section~3 in~\cite{Wang} for a graphical representation of the stochastic Magnus
expansion,~\cite{Muniz}~for higher strong order methods on matrix Lie groups,~\cite{Yang}~for
new Magnus-type methods for semi-linear non-commutative It\^o's SDEs
and~\cite{Kammdos} for numerical solutions of kinetic stochastic partial differential equations (SPDEs) via stochastic Magnus expansion.

In the sequel, we present two particular expressions of the stochastic exponent $(Y(t))_{t\geq 0}$ under some commutativity relations between the drift matrix $A$ and the noise dispersion (diffusivity) matrices $B_1,\ldots,B_{\ell}$.
In particular,

(i) for $\ell=1$ and $[A,B_1]=O_{n\times n}$, the Magnus expansion~\eqref{eq:shape} reads as follows
\begin{equation}\label{eq:martingale}
X(t;x)=\exp(At)M(t)x\quad \textrm{ for any }\quad x\in \mathbb{R}^n\quad \textrm{ and }\quad t\geq 0.
\end{equation}
where $M(t):=\exp(B_1W_1(t)-B^2_1 \frac{t}{2})$, $t\geq 0$.
We point out that the random process $(M(t))_{t\geq 0}$ is a martingale with respect to the natural filtration of the Brownian motion $W_1$. 
For $n=1$, it corresponds to the so-called risk-neutral measure in the  Black--Scholes--Merton model in mathematical finance, see Chapter~9 in~\cite{Mao}
and~\cite{Shreve}.
The expression~\eqref{eq:martingale} is also a particular case of~\eqref{ec:comu}.
The long-time behavior of~\eqref{eq:martingale} can be analyzed by the so-called Law of Iterated Logarithm. In fact, 
when $A<B^2_1/2$ one can note that an individual that hold for long-time the asset modeled by the Black--Scholes--Merton will be ruined almost surely, see Section~9.1 in~\cite{Mao} for details.

(ii) for $\ell=1$, $[A,B_1]\neq O_{n\times n}$ and $[[A,B_1],A]=[[A,B_1],B_1]=O_{n\times n}$ the Magnus expansion~\eqref{eq:shape} reads as follows
\begin{equation}\label{eq:martingaleno}
X(t;x)=\exp\left(
\left(A-\frac{B^2_1}{2}\right)t+B_1 W_1(t)+[B_1,A]\left(\frac{1}{2}tW_1(t)-\int_{0}^t W_1(s)\ud s\right)\right)x
\end{equation}
for any $x\in \mathbb{R}^n$ and $t\geq 0$.
We remark that the no commutativity assumption $[A,B_1]\neq O_{n\times n}$ turns out in the appearance of  a  random shifted \textit{integrated Brownian motion} $(I_t)_{t\geq 0}$ given by
\begin{equation}\label{eq:IBM}
I_t:=\frac{1}{2}tW_1(t)-\int_{0}^t W_1(s)\ud s,\quad
t\geq 0,
\end{equation}
in the  solution~\eqref{eq:martingaleno}. Using integration by parts formula and It\^o's isometry, one can verify that for each $t>0$, the random variable $I_t$ has a Gaussian distribution with zero mean and non-linear  (with respect to the time variable) variance $t^3/12$.

We also point out that the Magnus expansion of the solution to~\eqref{eq:modelito} is challenging even in dimension two. Indeed,
by the Table-Tennis Lemma (a.k.a.~Ping-Pong Lemma)
we note that particular $A$ and $B_1$, e.g.,~
\begin{equation}\label{eq:table}
A:=\begin{pmatrix}
1 & 2\\
0 & 1
\end{pmatrix}
\quad \textrm{ and }\quad
B_1:=\begin{pmatrix}
1 & 0\\
2 & 1
\end{pmatrix},
\end{equation}
generate the Sanov subgroup of the special linear group (under the usual multiplication of matrices)
 $\textsf{SL}(2,\mathbb{Z})$, which is free of rank two.
Since $A$ and $B_1$ are invertible,
we have that all iterated commutators of $A$ and $B_1$ are not zero,
see Example~25 on p.~26 of~\cite{Harpe}.

Since~\eqref{eq:shape} is complicated, the stability of the null solution can in general not be read easily from it. However, for a general $A\in \mathsf{M}_n$
the mean vector function $(m(t;x))_{t\geq 0}$, where $m(t;x):=\mathbb{E}[X(t;x)]$, $t\geq 0$,
satisfies the following vector differential equation
\begin{equation}\label{eq:meanequ}
\left\{
\begin{array}{r@{\;=\;}l}
\frac{\ud}{\ud t}m(t;x) & A\, m(t;x),\quad t\geq 0,\\
m(0;x)& x.
\end{array}
\right.
\end{equation}
In other words, $m(t;x)=e^{At}x$ for all $t\geq 0$ and $x\in \mathbb{R}^n$.
In addition, for  general $A,B_1,\ldots,B_\ell\in \mathsf{M}_n$ the autocorrelation matrix function $(U(t;x))_{t\geq 0}$, where $U(t;x):=\mathbb{E}[X(t;x)(X(t;x))^*]$, $t\geq 0$,
solves the following matrix differential equation
\begin{equation}\label{eq:autoequ}
\left\{
\begin{array}{r@{\;=\;}l}
\frac{\ud}{\ud t}U(t;x) & AU(t;x)+U(t;x)A^*+\sum\limits_{j=1}^{\ell}B_jU(t;x) B^*_j,\quad t\geq 0,\\
U(0;x)& xx^*,
\end{array}
\right.
\end{equation}
where $*$ denotes the transpose linear operator, see Theorem~3.2
in~\cite{Mao} or Theorem~(8.5.5) in~\cite{Arnold}.
We point out that using the usual \textsf{vec} operator, that is,
\[
\textsf{vec}(C):=(C_{1,1},\ldots,C_{1,n},C_{2,1},\ldots,C_{2,n},\ldots,C_{n,1},\ldots,C_{n,n})\in \mathbb{R}^{n^2}
\]
for $C=(C_{i,j})_{i,j\in\{1,\ldots,n\}}\in \mathsf{M}_n$,
one can
rewrite~\eqref{eq:autoequ} as a vector linear differential equation in $\mathbb{R}^{n^2}$.
More precisely,
\begin{equation}\label{eq:autoequvect}
\left\{
\begin{array}{r@{\;=\;}l}
\frac{\ud }{\ud t} \textsf{vec}(U(t;x)) & L\, \textsf{vec}(U(t;x)),\quad t\geq 0,\\
\textsf{vec}(U(0;x)))& \textsf{vec}(xx^*),
\end{array}
\right.
\end{equation}
where
\[
 \mathsf{M}_{n^2}\ni L:=A\otimes I_n+ I_n\otimes A+\sum_{j=1}^{\ell} B_j\otimes B^*_j,
\]
with $I_n$ denoting the identity matrix in $\mathsf{M}_n$ and $\otimes$ denoting the usual Kronecker product between matrices in $\mathsf{M}_n$,
see for instance Section~3.1 in~\cite{Buckwar} or Section~11.4 in~\cite{Arnold}.
Then the solution of~\eqref{eq:autoequvect} is given by
\[
\textsf{vec}(U(t;x))=e^{Lt}\textsf{vec}(xx^*)\quad \textrm{ for }\quad t\geq 0,\, x\in \mathbb{R}^n.
\]
We note that the null solution matrix $O_{n\times n}\in \textsf{M}_n$ of~\eqref{eq:autoequ} is asymptotically stable, if and only if the spectrum of the matrix $L$ is contained in the open left half-plane,
see Theorem~8.2 in~\cite{Hespanhabook}.

We say that the zero solution of~\eqref{eq:modelito} is exponentially $p$-stable for some $p>0$, if and only if, 
there exist positive constants $C:=C_p>$ and $\gamma:=\gamma_p$ satisfying
\[
\mathbb{E}\left[\|X(t;x)\|^p \right]\leq C\|x\|^p e^{-\gamma t}\quad \textrm{ for any  } \quad t\geq 0\quad \textrm{ and }\quad x\in \mathbb{R}^n,
\]
where $\|\cdot \|$ is the Euclidean norm in $\mathbb{R}^n$,
see Definition~\ref{def:pes} below.
When $p=2$, we say that the zero solution of~\eqref{eq:modelito} is exponentially mean-square stable.
Moreover, the zero solution of~\eqref{eq:modelito} is exponentially mean-square stable, 
if and only if, the zero matrix solution of the deterministic system~\eqref{eq:autoequ} is asymptotically stable,
see Item c) in (11.3.3)~Remark,  in~\cite{Arnold}.
Hence,~\eqref{eq:autoequ} provides a natural way to study mean-square stability. However, due to the lack
of a similar equation as~\eqref{eq:autoequ} for $p>0$, $p\neq 2$,
the study of the exponential $p$-stability for $p>0$, $p\neq 2$, (see Definition~\ref{def:pes} below) is \textit{a priori} not straightforward.

In this paper, we provide sufficient conditions for the
exponentially $p$-stable for some $p>0$
of the null solution of~\eqref{eq:modelito}.
The novelty and the main contribution of this paper is that such conditions can be rewritten as a Bilinear Matrix Inequality (BMI) problem.
Finally,  we exemplify our method in specific models from the literature.  More exactly, we consider random linear oscillators, satellite dynamics, inertia systems,
diagonal noise systems, off-diagonal noise systems, cancer self-remission modeling and modeling of smoking.
We emphasize that commutativity relations between the drift matrix and the noise matrices are \textit{not} assumed.

The rest of the manuscript is organized as follows. In Section~\ref{sec:pre} we provide the preliminaries and the main results of the manuscript. Next, in
Section~\ref{sec:examples}, we apply our main results to  specific models.
In Subsection~\ref{sec:dimtwo} we examine two dimensional models. More precisely, in Subsection~\ref{sec:rlo} we analyze a  random linear oscillator, in Subsection~\ref{sec:satelite} we study satellite dynamics, in Subsection~\ref{sec:inertia} we analyze a two-inertia system, in
Subsection~\ref{sec:diagonal} we study a diagonal noise system, and finally in
Subsection~\ref{sec:offdiagonal} we study off-diagonal noise systems.
In Subsection~\ref{sec:dimthree} we examine three dimensional systems.
More specific, in Subsection~\ref{sec:cancertumor}
we analyze a cancer self-remission model and in Subsection~\ref{sec:smookingstable} we study a smoking model. Finally, we give an appendix, which is divided in two sections. In Appendix~\ref{ap:proofmaintheorem} we prove Lemma~\ref{lem:lyapunov} and Theorem~\ref{th:stability} and in
Appendix~\ref{ap:sarkar} we prove Theorem~4.2 in~\cite{Sarkar}.

\section{\textbf{Preliminaries and main results}}\label{sec:pre}
\hfill

\noindent
The concept of stability of a given dynamic system
was introduced by A.~Lyapunov in~1892,
see~\cite{lyapunovbook1907endtrans}. Roughly speaking, it
measures the
robustness of a steady state of the system to small variations in the initial state or in the parameters of the system.
As for ordinary differential equations (for shorthand ODE), many stability properties for SDEs can be
analyzed and characterized by the existence of so-called Lyapunov functions,
see for instance the monograph~\cite{Shaikhet} for further details.
Since Lyapunov functions can in general not be computed analytically, one usually must resort to computational methods, cf.~e.g.~\cite{GieslHafstein,Khasminskii}, Chapter~8 in~\cite{NeckelRupp} and~\cite{Rupp}.
The stability can be also analyzed using the so-called (upper) Lyapunov exponents, nevertheless, in general there are no formulas for such exponents, for further discussions see for instance~\cite{Arnold,Imkeller,Mao,Talay}.

By It\^o's formula, see for instance Chapter~4
in~\cite{Mao},
the generator $\mathcal{L}$ of~\eqref{eq:modelito}  acts on a scalar smooth observable $V\in \mathcal{C}^2(\mathbb{R}^n,\mathbb{R})$
as follows
\begin{equation}\label{eq:lyapunov}
(\mathcal{L}V)(z)=\langle \nabla V(z),Az\rangle+\frac{1}{2}\sum_{k=1}^{\ell} z^* B^*_k \textsf{Hess}(V(z))B_k z\quad \textrm{ for any }\ \ z\in \mathbb{R}^n,
\end{equation}
where $\textsf{Hess}(V(z))$ denotes the Hessian matrix of $V$ at the point $z$, $\langle\cdot,\cdot\rangle$ denotes the standard inner product of $\mathbb{R}^n$ and recall that $*$ denotes the transpose linear operator.
The generator~\eqref{eq:lyapunov} reads in the canonical coordinates as follows
\begin{equation}
(\mathcal{L}V)(z)=\sum_{i,j=1}^{n} A_{i,j}z_{j} \frac{\partial}{\partial z_{i}} V(z)+\frac{1}{2}\sum_{k=1}^{\ell}\sum_{i_1,i_2,j_1,j_2=1}^{n}
(B_k)_{i_1,j_1}
(B_k)_{i_2,j_2}z_{j_1} z_{j_2}
 \frac{\partial^2}{\partial z_{i_1} \partial z_{i_2}}V(z),
\end{equation}
for any $z\in \mathbb{R}^n$.
The zero element of $\mathbb{R}^n$ is denoted by $0_\textsf{n}$.
By Lemma~5.3 in~\cite{Khasminskii} we have that $0_n$ is
inaccessible to any sample-path of the
process~\eqref{eq:modelito} with initial condition $x\neq 0_n$. Therefore, we can still apply the
generator~\eqref{eq:lyapunov} for smooth observables $V$ that are twice continuously differentiable except perhaps at the point $0_n$.
It is easy to see  that the null solution (a.k.a. trivial solution or equilibrium position) $X(t;0_\textsf{n})\equiv 0_\textsf{n}$ for all $t\geq 0$ is a fixed point (the Dirac mass at zero $\delta_0$ is a stationary law) of the stochastic dynamics given
by~\eqref{eq:modelito}.

We start recalling the definitions of Stability in Probability (SiP), 
Asymptotically Stable in Probability (ASiP),
Global Asymptotic Stability in Probability (GASiP) and the Exponential $p$-Stability (p-ES).
We recall that $\|\cdot \|$ is the Euclidean norm induced by the standard inner product $\langle\cdot,\cdot\rangle$.

\begin{definition}[Stability in Probability (SiP)]
\hfill

\noindent
The null solution $(X(t;0_\textsf{n}))_{t\geq 0}$
of~\eqref{eq:modelito} is stable in probability if and only if for any $r>0$ and $\varepsilon>0$ there exists $\delta>0$ such that
\[
\mathbb{P}\left(\sup\limits_{t\geq 0}\|X(t;x)\|\leq r\right)>1-\varepsilon \quad \textrm{ whenever }\quad \|x\|\leq \delta.
\]
\end{definition}

\begin{definition}[Asymptotically stable in probability (ASiP)]
\hfill

\noindent
The null solution $(X(t;0_\textsf{n}))_{t\geq 0}$
of~\eqref{eq:modelito} is asymptotically stable in probability if and only if for any $\varepsilon>0$  there exists $\delta=\delta_\varepsilon>0$ such that
\[
\mathbb{P}\left(\lim\limits_{t\to \infty}\|X(t;x)\|=0\right)>1-\varepsilon \quad \textrm{ whenever }\quad \|x\|\leq \delta.
\]
\end{definition}

\begin{definition}[Global Asymptotic Stability in Probability (GASiP)]
\hfill

\noindent
The null solution $(X(t;0_\textsf{n}))_{t\geq 0}$
of~\eqref{eq:modelito} is globally asymptotically stable in probability if and only if it is SiP and for any initial condition $x\in \mathbb{R}^n$ it follows that
\[
\mathbb{P}\left(\lim\limits_{t\to \infty}\|X(t;x)\|=0\right)=1.
\]
\end{definition}

\begin{definition}[Exponential $p$-Stability (p-ES)]\label{def:pes}
\hfill

\noindent
The null solution $(X(t;0_\textsf{n}))_{t\geq 0}$
of~\eqref{eq:modelito} is exponentially $p$-stable for some $p>0$ if and only if there exist positive constants $C:=C_p$ and $\gamma:=\gamma_p$ such that
\[
\mathbb{E}\left[\|X(t;x)\|^p \right]\leq C\|x\|^p e^{-\gamma t}\quad \textrm{ for any  } \quad t\geq 0\quad \textrm{ and }\quad x\in \mathbb{R}^n.
\]
\end{definition}

For the linear SDE~\eqref{eq:modelito} we have  that GASiP is equivalent to  ASiP. Moreover, $p$-ES for any $p>0$ implies GASiP. For details, we refer to Proposition~1 in~\cite{HafsteinGudmundsson}.
For others concepts of stability for the null solution of~\eqref{eq:modelito}, we refer to Definition~2 and Remark~3 in~\cite{HafsteinGudmundsson}.

Let $A\in \mathsf{M}_n$ be fixed. We recall that the matrix $A$ is Hurwitz stable, if and only if the spectrum of $A$ (the set of the eigenvalues of $A$) is contained in $\mathbb{C}^{-}:=\{z\in \mathbb{C}:\mathsf{Re}(z)<0\}$.
We recall that the set of Hurwitz matrices is not a convex set and that  it is not closed under the usual addition and the usual multiplication of matrices.

We start with the following observation.
The autonomous linear differential system $\dot X(t;x)=AX(t;x)$, $t\geq 0$
is asymptotically stable,
if and only if the matrix $A$ is Hurwitz stable.
It is well-known that then it is possible to construct a classical Lyapunov function $V:\mathbb{R}^n\to [0,\infty)$ of the form $V(x)=x^*Qx=:\|x\|^2_{Q}$, where $Q\in \mathsf{M}_n$ is the unique symmetric and positive definite solution of the so-called Matrix Lyapunov Equation
\[
A^*Q+QA=-I_n\quad \textrm{ with } I_n \textrm{ being the identity in}\ \textsf{M}_n.
\]
Even though $Q=\int_0^\infty e^{A^*s}e^{A s}\ud s$, we point out  that \textit{generically}  there is no explicit formula for $Q$. For the particular case $A^*=A$ we have that $Q=-\frac{1}{2}A^{-1}$.
For further details, see~\cite{Horn,Lancaster,Liu,Simoncini}.

In the sequel, we introduction a notation for comparison of matrices.
Let $A,B\in \mathsf{M}_n$ and assume that $A$ and $B$ are symmetric. We write
\begin{equation}\label{eq:geq}
A\succeq B \quad \textrm{ whenever }\quad x^* Ax\geq x^* B x \quad \textrm {for all }\quad x\in \mathbb{R}^n.
\end{equation}
If $A$ is not symmetric, we note that
the antisymmetric part $(1/2)x^*(A-A^*)x=0$ for all $x\in \mathbb{R}^n$. Hence, the notation~\eqref{eq:geq} can also be applied to non-symmetric matrices $A$ and $B$ replacing them by
$(A+A^*)/2$ and $(B+B^*)/2$, respectively.  We define $\preceq$ in the obvious way, i.e., $A \preceq B$ when $-A \succeq -B$.

Now, we state a criterion for verifying GASiP. The proof is given in Theorem~3.2 and Theorem~3.3
of~\cite{HafsteinGudmundsson} or Theorem~5.11 and Theorem~6.2 in~\cite{Khasminskii}.
One of the largest advantages of the Lyapunov method is that an explicit solution to~\eqref{eq:modelito} is not needed, which can be very complicated as we discussed in Section~\ref{sub:MBM}.

\begin{theorem}[Lyapunov's criterium]\label{th:tool}
\hfill

\noindent
The null solution $(X(t;0_\textsf{n}))_{t\geq 0}$
of~\eqref{eq:modelito} is GASiP, if and only if for a $p>0$ there exist positive constants $c_1,c_2,c_3,c_4,c_5$, a scalar function $V\in \mathcal{C}(\mathbb{R}^n, \mathbb{R})\cap \mathcal{C}^2(\mathbb{R}^n\setminus \{0_\textsf{n}\},\mathbb{R})$ and $Q\succ 0$ (symmetric and positive definite square
matrix of dimension $n\times n$)  satisfying
\begin{itemize}
\item[(i)]
$c_1\|x\|^p_Q\leq V(x)\leq c_2 \|x\|^p_Q$
for all  $x\in \mathbb{R}^n$,
\item[(ii)] $(\mathcal{L}V)(x)\leq -c_3 \|x\|^p_Q$ for all $x\in \mathbb{R}^n\setminus \{0_\textsf{n}\}$,
\item[(iii)] the function $V$ is positively homogeneous of degree $p$, that is, $V(\lambda x)=|\lambda|^p V(x)$ for any $x\in \mathbb{R}^n$ and $\lambda>0$.
\item[(iv)] the following uniform bound on the partial derivatives holds true:
\[
\max\limits_{k\in \{1,\ldots,n\}}\left|\frac{\partial }{\partial x_k}V(x)\right|\leq c_4 \|x\|^{p-1}_Q\quad \textrm{ for any }\quad x\in \mathbb{R}^n\setminus \{0_\textsf{n}\},
\]
\item[(v)]
the following uniform bound on the second partial derivatives holds true:
\[
\max\limits_{j,k\in \{1,\ldots,n\}}\left|\frac{\partial^2 }{\partial x_jx_k}V(x)\right|\leq c_5 \|x\|^{p-2}_Q\quad \textrm{ for any }\quad x\in \mathbb{R}^n\setminus \{0_\textsf{n}\},
\]
\end{itemize}
where $\|x\|_Q:=(x^*Qx)^{1/2}$ for all $x\in \mathbb{R}^n$.
\end{theorem}
The function $V$ in Theorem~\ref{th:tool} is called a Lyapunov function for~\eqref{eq:modelito}. It can be interpreted in physical terms as a potential energy.
Since the noise in~\eqref{eq:modelito} is multiplicative, the so-called stabilization by noise phenomenon naturally occurs. However, we point out that for the linear SDE~\eqref{eq:modelito}, of which the deterministic dynamics is not stable, there cannot exists a globally twice differentiable Lyapunov function, see Remark~5.5 in~\cite{Khasminskii}.

\begin{lemma}[Lyapunov function]\label{lem:lyapunov}
\hfill

\noindent
The null solution $(X(t;0_\textsf{n}))_{t\geq 0}$
of~\eqref{eq:modelito} is GASiP if
for some $p>0$, $c>0$ and positive definite square matrix $Q\in \mathsf{M}_n$,  we have
\begin{equation}\label{eq:Hine}
H(x)\geq c \|x\|^4 \quad \textrm{ for all }\quad x\in \mathbb{R}^n,
\end{equation}
where 
\begin{align}\label{cond:H}
\begin{split}
H(x):&=-x^* \left( A^* Q+  QA+ \sum_{j=1}^\ell B_j^* Q B_j\right)x \|x\|_Q^2\\
&\qquad + \frac{2-p}{4}\sum_{j=1}^\ell \left(x^* (QB_j+B_j^* Q)x\right)^2, \quad x\in \mathbb{R}^n.
\end{split}
\end{align}
with $\|x \|$ being the Euclidean norm of $x\in \mathbb{R}^n$,
and $\|x\|_Q=(x^*Qx)^{1/2}$ for all $x\in \mathbb{R}^n$.
In fact, if~\eqref{eq:Hine} holds true, then $V(x)=\|x\|_Q^p$ is a Lyapunov function for~\eqref{eq:modelito} and the null solution is $p$-ES.
\end{lemma}

The proof of Lemma~\ref{lem:lyapunov} is given in Subsection~\ref{sub:prueba} of Appendix~\ref{ap:proofmaintheorem}.

\begin{remark}[LMI and BMI problems]
\label{rem:lmip2}\hfill

\noindent
A short inspection of~\eqref{cond:H} reveals the following fact:
if for a particular $p=p^*>0$ we have $H(x) \geq c \|x\|^4$ for all $x\in\mathbb{R}^n$ in Lemma~\ref{lem:lyapunov}, then it also holds true for all $p>0$ fulfilling $p\le p^*$.

If $p=2$ in Lemma~\ref{lem:lyapunov} we have a simple linear matrix inequality condition (for shorthand LMI) that is equivalent to $H(x) \ge c \|x\|^2$, $x\in\mathbb{R}^n$.
Indeed, since $\|x\|_Q^2\ge \lambda_{\min}(Q)\|x\|^4$, where $\lambda_{\min}(Q)>0$ denotes the smallest eigenvalue of the positive and symmetric matrix $Q\in\mathsf{M}_n$, we
have that $H(x) \ge c \|x\|^4$, $x\in \mathbb{R}^n$ for some appropriate $c>0$, if and only if the LMI
\begin{align}\label{LMIeq}
Q \succeq \varepsilon I_n\quad  \textrm{and}\quad  A^* Q+  QA+ \sum_{j=1}^\ell B_j^* Q B_j \preceq -\varepsilon I_n
\end{align}
has a solution for some $\varepsilon >0$. LMIs can be solved algorithmically~\cite{Boyd,S98guide} and~\eqref{LMIeq} has a solution, if and only if, the null solution 
to~\eqref{eq:modelito} is mean-square stable, see 
Corollary~11.4.14 in~\cite{Arnold}.
Hence,  Lyapunov functions can be easily constructed for systems with a mean-square stable equilibrium.
 For this reason $2$-ES, often called exponential mean-square stability,
is often encountered in the literature and we discuss this further in the examples in
Section~\ref{sec:examples}. We stress that the null solution
of~\eqref{eq:modelito} is exponentially mean-square stable, if
and only if, the LMI~\eqref{LMIeq} has a solution.
The limitation of using $p=2$ is that many systems, of which the null solution is GASiP, are not exponentially mean-square stable, see also Remark~\ref{p2remark2}.

In general, i.e.,~for $p\neq 2$ and in particular for $0<p<2$, the conditions for $H(x) \ge c \|x\|^4$ can be written as a bilinear  matrix inequality (for shorthand BMI, see e.g.~\cite{VanAntwerp}), and this is discussed in detail in Section~\ref{sec:Mrc} below.
To the best of the authors knowledge, there are no algorithms known  to solve BMIs.  However, several different heuristics have been developed.  
In the examples in  Section~\ref{sec:examples} we use the heuristics from~\cite{KZM2018BMI1,KZM2018BMI2},
which  solves a sequence of LMIs in order to find a feasible solution to the BMI problem.
These LMIs are penalized convex relaxations of the original non-convex BMI problem and
convex relaxations rely on convex quadratic constraints.  This heuristic has been shown to be  effective in numerous benchmark problems; for a discussion of our implementation of this heuristic see~\cite[Chapter 4]{Bjarkason}.
\end{remark}

\subsection{\textbf{Main results and consequences}} \label{sec:Mrc}
\hfill

\noindent
In this subsection, we establish the main results and their consequences. We start with some preliminaries.

\begin{definition}[Bilinear Matrix Inequality (BMI)]\label{def:BMI}
\hfill

\noindent
Let $N\in \mathbb{N}$ and let
$A_{(i,j)},B_{(i)},C\in \mathsf{M}_N$ for $i,j=1,\ldots,M$ be symmetric matrices.
 A bilinear matrix inequality is defined as
\begin{equation}\label{eq:BMIP}
J:=\sum_{i=1}^{M}\sum_{j=1}^{M} q_iq_jA_{(i,j)}+\sum_{i=1}^{K} q_{i}B_{(i)}+C\succeq O_{N \times N}
\end{equation}
for some real numbers $(q_{i})_{i=1}^{\max\{K,M\}}$ called the variables of the BMI, where $O_{N \times N}$ denotes the zero matrix in $\mathsf{M}_N$.
\end{definition}

We note that the entries of the matrix $J$ defined in the left-hand side of~\eqref{eq:BMIP} are second order polynomials in the variables $q_1,\ldots,q_M$.
We point out that we can always choose $K=M$ by adding zeros, i.e., $q_j=0$ for $j\in \{\min\{K,M\}+1,\ldots,\max\{K,M\}\}$.

The following theorem is the main result of this manuscript.
Roughly speaking, it allows us to compute a symmetric and positive definite $Q\in \mathsf{M}_n$ for  Item~(ii) in Theorem~\ref{th:tool} as a BMI feasibility problem, such that  $V(x)=\|x\|^p_Q$, $x\in \mathbb{R}^n$, is a Lyapunov function for~\eqref{eq:modelito}.
The number $c>0$ and the entries of the symmetric and positive definite matrix $Q$ will be variables of the BMI problem~\eqref{eq:BMIP}.  For example, one can let $q_1,q_2,\ldots,q_{n(n+1)/2}$ be the components in the upper triangle of $Q$ and set $q_K:=c$.
In general, some auxiliary variables are needed in the BMI problem so $K>n(n+1)/2+1$, 
see~\eqref{eq:defK} in Appendix~\ref{ap:proofmaintheorem} for a possible choice of $K$.
\begin{theorem}[GASiP as a BMI problem]\label{th:stability}
\hfill

\noindent
The inequality~\eqref{eq:Hine} can be algorithmically rewritten as a BMI problem with variables $c$ and $Q$, see above.  A feasible solution to the BMI problem delivers a symmetric and positive definite $Q\in\mathsf{M}_n$ such that $V(x)=\|x\|_Q^p$, $x\in \mathbb{R}^n$, is a Lyapunov function
for the system~\eqref{eq:modelito}.
\end{theorem}
The proof of Theorem~\ref{th:stability} is given in Subsection~\ref{sec:proofthnew} of Appendix~\ref{ap:proofmaintheorem}.
We emphasize that the matrices $A, B_1\ldots, B_{\ell}$ in~\eqref{eq:modelito} are time-independent.

\begin{remark}[One-dimensional Geometric Brownian motion]
\hfill

\noindent
The stability problem for the null solution for
one-dimensional systems ($n=1$) of the form~\eqref{eq:modelito} is straightforward and can be analyzed explicitly, see for instance Example~4.10 and Example~5.5 in~\cite{Mao}, and Remark~5.5 in~\cite{Khasminskii}.
The reason for this is that $1\times 1$ matrices commute, which renders the solution formula tractable.
\end{remark}

From here to the end of the manuscript, we always assume that $n\in \mathbb{N}\setminus\{1\}$.

\begin{remark}[Multivariate Geometric Brownian motion with specific structure]
\hfill

\noindent
For the case
$A:=\widetilde{A}+D$,
where $\widetilde{A}\in\mathsf{M}_n$ is  antisymmetric, $D\in\mathsf{M}_n$ is a diagonal matrix
and $B_j\in \textsf{Span}(I_n)$ for all $j\in \{1,\ldots,\ell\}$, it is not difficult to find explicit conditions for the null solution being GASiP, see for instance Example~6.2 in~\cite{Khasminskii}.
\end{remark}

\begin{remark}[The particular choice $p=2$ drops important terms in $H$] \label{p2remark2}
\hfill

\noindent
We note that for Hurwitz stable matrix $A$ and a symmetric and negative definite matrix $N$, there exists a unique symmetric and positive definite $Q$ such that  $A^*Q+  QA=N$, see Theorem~1, p.~443
in~\cite{Lancaster}. Therefore, the first term in the right-hand side of~\eqref{cond:H} is contracting for the dynamics~\eqref{eq:modelito}. However, in the  case  $-A$ is Hurwitz stable, the quadratic form
\[
x^*\left(A^* Q+  QA+ \sum_{j=1}^\ell B_j^* Q B_j\right)x
\]
is positive and hence this term is pushing away from zero. We observe that the second term of~\eqref{cond:H} has a contracting behavior for~\eqref{eq:modelito} when $p\in (0,2)$.

We point out that for a suitable $p>0$ the goal is to find a positive constant $c$ such that
\[
H(x)\geq c\|x\|^4_Q\quad \textrm{ for all }\quad x\in \mathbb{R}^n,
\]
which yields
\begin{equation}
(\mathcal{L}V)(x)\leq -\frac{p}{2}c V(x)\quad \textrm{ for all }\quad x\in \mathbb{R}^n.
\end{equation}
\end{remark}

\begin{remark}[Some properties of $H$]\label{rem:multihomog}
\hfill

\noindent
Using the multi-index notation, we observe that~\eqref{cond:H} can be written using real coefficients $q_\alpha$ as
\begin{equation}\label{eq:Hmulti}
H(x)=\sum\limits_{|\alpha|=4} q_\alpha x^\alpha\quad \textrm{ for all }\quad x\in \mathbb{R}^n.
\end{equation}
We also note that $H$ is a homogeneous polynomial of degree four, that is, $H(\lambda x)=\lambda^4 H(x)$ for all $\lambda \in \mathbb{R}$ and $x\in \mathbb{R}^n$.
In addition, the coefficients $q_{\alpha}$ are  second order polynomials with respect to the coefficients of $Q$ and they can be  computed explicitly from the coefficients of $A$ and $B_1,\ldots,B_{\ell}$. The latter is straightforward by computer software.
\end{remark}

\begin{remark}[$P_c$ as a sum of squared  polynomials]
\hfill

\noindent
We observe that~\eqref{eq:Hine} can be written using a multivariate polynomial in the variables $x=(x_1,\ldots,x_n)$ as follows
\begin{equation}\label{eq:Pc}
P_c(x)=P_{c}(x_1,\ldots,x_n):=H(x_1,\ldots,x_n)-c\left(\sum\limits_{j=1}^{n}x^2_j\right)^2\geq 0.
\end{equation}
In general, a positive multivariate polynomial
cannot be decomposed as a sum of suitable squared polynomials,
see Hilbert's seventeenth problem~\cite{Hilbert, Reznick1,Reznick2}  and the monographs~\cite{Marshall,Powers} for further details about theory and practice of positivity of real polynomials.  The exceptions are when a multivariate polynomial is homogeneous of order two or if $n=3$ and it is homogenous of
order four.
However, since it is a very hard problem to determine if a polynomial is positive, the sum of squared polynomials relaxation is often used in practice and usually works
very well.
We will show that we can algorithmically write down a BMI feasibility problem for~\eqref{eq:Pc}, such that $P_c$ can be written as a sum of suitable squared polynomials from its solution.
To be more precise, our goal is to \textbf{find} for an a\,priori fixed $\varepsilon>0$
\begin{itemize}
\item[(a)]
a constant $c\ge \varepsilon$,
\item[(b)] a positive definite symmetric matrix $Q\in \mathsf{M}_n$, $Q \succeq \varepsilon I_n$
\item[(c)] and a symmetric and positive semi-definite matrix $P\in \mathsf{M}_m$
\end{itemize}
such that
\begin{equation}\label{eq:Pz}
P_c(x)=z^*Pz \quad \textrm{ for all }\quad z\in \mathbb{R}^m,
\end{equation}
where $m:=n(n+1)/2$ and
\begin{equation}\label{def:z}
z:=(x^2_1,x_1x_2,x_1x_3,\ldots, x_1x_{n},x^2_2,x_2x_3,\ldots,x_2x_n,x^2_3, x_3x_4,\ldots,x_3x_n,\ldots,x^2_{n-1}, x_{n-1}x_n, x^2_n)^*.
\end{equation}
Note that because the conditions $>0$ and $\succ O_{n \times n}$ cannot be implemented in a semi-definite optimization problem, we need a small constant $\varepsilon>0$ and implement
$\ge \varepsilon$ and $\succeq \varepsilon I_n$ instead.
Further, note that the entries of the vector $z$ consists of second order monomials, hence generically $z^*Pz$ is a multivariate polynomial of fourth order.
Since the $P$ in~\eqref{eq:Pz} is a symmetric and positive semi-definite matrix, all its eigenvalues $d_1,\ldots,d_m$ are non-negative.
Then the celebrated spectral theorem yields
\[
P_c(x)=\sum_{j=1}^{m} (\sqrt{d_j}[Oz]_j)^2,
\]
where $O\in \mathsf{M}_m$ is an orthogonal matrix and $Oz=([Oz]_1,\ldots,[Oz]_m)^*$, i.e., $[Oz]_k$ denotes the $k$-th coordinate of $Oz$ with respect to the canonical basis $(e_j)_{1\leq j\leq m}$.
\end{remark}

\begin{remark}[The computation of the entries for the matrix $P$]\label{rem:compuentries}
\hfill

\noindent
By~\eqref{eq:Pz} we have that
\begin{equation}\label{eq:Pcrelation}
P_c(x_1,\ldots,x_n)=z^*Pz\quad \textrm{ for all }\quad x\in \mathbb{R}^n,
\end{equation}
where $z=z(x_1,\ldots,x_n)\in \mathbb{R}^m$ is defined in~\eqref{def:z}. We compute all  fourth order  partial derivatives on both sides of~\eqref{eq:Pcrelation} and obtain linear relations between the coefficients of the matrix $P$ and the coefficients of the polynomial $P_c$ given in~\eqref{eq:Pc}.
We remark that such linear relations do not determine the elements of $Q$ uniquely, only some linear relations they must fulfill.
Further,  note that the coefficients of $P$ depend on $c$ and the coefficients of $Q$.
\end{remark}

\subsection{\textbf{Construction of the BMI problem for $n=2$ and $\ell=1$}}
\hfill

\noindent
In other to be clear in the presentation, we give the details when $n=2$, ($m=3$) and $\ell=1$.
The general case is given in detail in the proof of Theorem~\ref{th:stability}.

For $n=2$ (so $m=3$), $\ell=1$ the coefficients of the symmetric matrix $P=(P_{i,j})_{i,j\in \{1,2,3\}}$ given  $Q=(Q_{i,j})_{i,j\in \{1,2\}}$ and $c$, and  $A=(A_{i,j})_{i,j\in \{1,2\}}$, $B=(B_{i,j})_{i,j\in \{1,2\}}$ in~\eqref{eq:modelito},
satisfy~\eqref{eq:Pcrelation}, that is,
\begin{equation}\label{eq:defPcdos}
\begin{split}
P_c(x_1,x_2)&=
\begin{pmatrix}
x^2_1 & x_1x_2 & x^2_2
\end{pmatrix}
\begin{pmatrix}
P_{1,1} & P_{1,2} & P_{1,3}\\
P_{2,1} & P_{2,2} & P_{2,3}\\
P_{3,1} & P_{3,2} & P_{3,3}
\end{pmatrix}
\begin{pmatrix}
x^2_1\\
x_1x_2\\
x^2_2
\end{pmatrix}\\
&=\begin{pmatrix}
x^2_1 & x_1x_2 & x^2_2
\end{pmatrix}
\begin{pmatrix}
P_{1,1} & P_{1,2} & P_{1,3}\\
P_{1,2} & P_{2,2} & P_{2,3}\\
P_{1,3} & P_{2,3} & P_{3,3}
\end{pmatrix}
\begin{pmatrix}
x^2_1\\
x_1x_2\\
x^2_2
\end{pmatrix}\\
&=P_{1,1}x^4_1+2P_{1,2}x^3_1x_2+(2P_{1,3}+P_{2,2})x^2_1x^2_2
+2P_{2,3}x_1x^3_2+P_{3,3}x^4_2
\end{split}
\end{equation}
for all $(x_1,x_2)\in \mathbb{R}^2$.
On the other hand, by~\eqref{cond:H}
and~\eqref{eq:Pc} we have
\begin{equation}\label{eq:Pctwo}
P_c(x_1,x_2)=u_1x^4_1+u_2x^3_1x_2+u_3x^2_1x^2_2
+u_4x_1x^3_2+u_5x^4_2
-c(x^2_1+x^2_2)^2
\end{equation}
for all $(x_1,x_2)\in \mathbb{R}^2$, where the coefficients $u_1,u_2,u_3,u_4,u_5$ can be computed explicitly and they are functions of the coefficients of $Q$, and implicitly of the coefficients of $A$, $B$ and the constants $c$ and $p$.
By~\eqref{eq:defPcdos} we have
\begin{equation}
\begin{split}
\frac{\partial^4 P_c(x_1,x_2)}{\partial x^4_1}&=24P_{1,1},\\
\frac{\partial^4 P_c(x_1,x_2)}{\partial x^3_1x_2}&=12P_{1,2},\\
\frac{\partial^4 P_c(x_1,x_2)}{\partial x^2_1x^2_2}&=8P_{1,3} + 4P_{2,2},\\
\frac{\partial^4 P_c(x_1,x_2)}{\partial x_1x^3_2}&=12P_{2,3},\\
\frac{\partial^4 P_c(x_1,x_2)}{\partial x^4_2}&=24P_{3,3}.
\end{split}
\end{equation}
By computing the fourth order partial derivatives
of~\eqref{eq:defPcdos} and~\eqref{eq:Pctwo}   we obtain that
\begin{equation}
\begin{split}
P_{1,1}, P_{1,2}, 2P_{1,3}+P_{2,2}, P_{2,3}, P_{3,3}&=\textrm{ explicit linear combination of }\\
&\qquad Q^2_{1,1}, Q_{1,1}Q_{1,2}, Q_{1,1}Q_{2,2}, Q^
2_{1,2}, Q_{1,2}Q_{2,2}, Q^2_{2,2} \textrm{ and }  c,
\end{split}
\end{equation}
see p.~27 and 28 in~\cite{Bjarkason} or
Section~4 for the linear oscillator in \cite{HafsteinBMI}.
More precisely,
\begin{equation}\label{eq:defP11}
\begin{split}
P_{1,1}&=(-2A_{1,1}+(1-p)B^2_{1,1})Q^2_{1,1}\\
&\quad + (2B_{1,1}B_{2,1}(1-p)-2A_{2,1})Q_{1,1}Q_{1,2}\\
&\quad +(-B^2_{2,1})Q_{1,1}Q_{2,2}\\
&\quad+(B^2_{2,1}(2-p))Q^2_{1,2}\\
&\quad +(0) Q_{1,2}Q_{2,2}\\
&\quad+(0)Q^2_{22}\\
&\quad +
(-1)c,
\end{split}
\end{equation}
\begin{equation}\label{eq:defP12}
\begin{split}
P_{1,2}&=(-2A_{1,2}+B_{1,1}B_{1,2}(1-p))Q^2_{1,1}\\
&\quad+ (-3A_{1,1}-A_{2,2}-B^2_{1,1}-B_{1,1}B_{2,2}\\
&\qquad\quad + B_{1,2}B_{2,1}(1-p)+B_{1,1}(B_{1,1} + B_{2,2})(2-p))Q_{1,1}Q_{1,2}\\
&\quad+(-A_{2,1}-B_{2,1}B_{2,2}+B_{1,1}B_{2,1}(2-p))Q_{1,1}Q_{2,2}\\
&\quad+(-2A_{2,1}-2B_{1,1}B_{2,1}+B_{2,1}(B_{1,1} + B_{2,2})( 2-p))Q^2_{1,2}\\
&\quad +(B^2_{2,1}(1 - p)) Q_{1,2}Q_{2,2}\\
&\quad+(0)Q^2_{22}\\
&\quad +
(0)c,
\end{split}
\end{equation}
\begin{equation}
\begin{split}\label{eq:defP13P22}
2P_{1,3} + P_{2,2}&=(B_{1,2}^2(1-p))Q^2_{1,1}\\
&\quad+ (-6A_{1,2}+4B_{1,1}B_{1,2}(1-p)+2B_{1,2}B_{2,2}(1-p)
)Q_{1,1}Q_{1,2}\\
&\quad+(-2A_{1,1} - 2A_{2,2} - B_{1,1}^2 - B_{2,2}^2\\
&\qquad \quad
+2B_{1,1}B_{2,2}(2-p) + 2B_{1,2}B_{2,1}(2-p))Q_{1,1}Q_{2,2}\\
&\quad+(-4A_{1,1} - 4A_{2,2} + (2-p)(B_{1,1} + B_{2,2})^2 - 4B_{1,1}B_{2,2} -2p B_{1,2}B_{2,1})Q^2_{1,2}\\
&\quad +(-6A_{2,1}+2B_{1,1}B_{2,1}(1-p) + 4B_{2,1}B_{2,2}(1-p)
) Q_{1,2}Q_{2,2}\\
&\quad+(B_{2,1}^2(1-p))Q^2_{22}\\
&\quad +
(-2)c,
\end{split}
\end{equation}
\begin{equation}\label{eq:defP23}
\begin{split}
P_{2,3}&=(0)Q^2_{1,1}\\
&\quad+ (B_{1,2}^2(1-p))Q_{1,1}Q_{1,2}\\
&\quad+(-A_{1,2}-B_{1,1}B_{1,2}+B_{1,2}B_{2,2}(2-p))Q_{1,1}Q_{2,2}\\
&\quad+(-2A_{1,2}-2B_{1,2}B_{2,2}+B_{1,2}(B_{1,1} + B_{2,2})(2-p))Q^2_{1,2}\\
&\quad +(-A_{1,1}-3A_{2,2}-B_{2,2}^2-B_{1,1}B_{2,2}\\
&\qquad\quad
+B_{1,2}B_{2,1}(1-p)+B_{2,2}(B_{1,1} + B_{2,2})(2-p)) Q_{1,2}Q_{2,2}\\
&\quad+(-A_{2,1} + B_{2,1}B_{2,2}(1-p))Q^2_{2,2}\\
&\quad +
(0)c,
\end{split}
\end{equation}
and finally
\begin{equation}\label{eq:defP33}
\begin{split}
P_{3,3}&=(0)Q^2_{1,1}\\
&\quad+ (0)Q_{1,1}Q_{1,2}\\
&\quad+(-B_{1,2}^2)Q_{1,1}Q_{2,2}\\
&\quad+(B_{1,2}^2(2-p))Q^2_{1,2}\\
&\quad +(2B_{1,2}B_{2,2}(1-p)-2A_{1,2}) Q_{1,2}Q_{2,2}\\
&\quad+(-2A_{2,2}+B_{2,2}^2(1-p))Q^2_{22}\\
&\quad +
(-1)c.
\end{split}
\end{equation}
We note that the coefficients $P_{1,3}$ and $P_{2,2}$ are not uniquely determined. They only satisfy the linear dependency relation
\begin{equation}\label{eq:2p}
\begin{split}
2P_{1,3}+P_{2,2}&=\alpha_1Q^2_{1,1}+\alpha_{2}Q_{1,1}Q_{1,2}+\alpha_3 Q_{1,1}Q_{2,2}\\
&\quad+\alpha_4 Q^2_{1,2}+\alpha_5 Q_{1,2}Q_{2,2}+\alpha_6 Q^2_{2,2}+\alpha_7 c
\end{split}
\end{equation}
for some explicit coefficients $\alpha_1,\ldots,\alpha_7$ defined in~\eqref{eq:defP23}. Since the final goal is to find a positive definite matrix $Q$ such that~\eqref{cond:H} is positive for all $x\in \mathbb{R}^2$,
 we stress that~\eqref{eq:2p} provides us flexibility for finding such matrix $Q$.
We now find real numbers $q_1,q_2,q_3,q_4,q_5,q_6$ and matrices $A_{(i,j)},B_{(i)},C\in \mathbb{R}^{8\times 8}$
such that
\begin{equation}\label{eq:BMIdim2}
\sum_{i=1}^{3}\sum_{j=i}^{3} q_iq_jA_{(i,j)}
+\sum_{i=1}^{6} q_i B_{(i)}+C\succeq O_{8\times 8}.
\end{equation}
We start with the assignment to the variables $(q_j)_{j\in \{1,2,3,4,5,6\}}$. We set
\begin{equation}\label{defn:qs}
q_1:=Q_{1,1},\quad q_2:=Q_{1,2},\quad q_3:=Q_{2,2}
\quad \textrm{ and } \quad q_6=c.
\end{equation}
In addition, we only have one dependent relation~\eqref{eq:2p} and hence we assign
\begin{equation}\label{defn:otqs}
q_4:=P_{1,3}\quad \textrm{ and }\quad q_5:=P_{2,2}.
\end{equation}

For each $i,j\in \{1,2,\ldots,8\}$ let $E_{(i,j)}\in \mathsf{M}_8$ such that all entries of $E_{(i,j)}$ are equal to zero except in the
$(i,j)$ and $(j,i)$ entries which are equal to one.
We first observe that
for each $i,j\in \{1,2,\ldots,8\}$
the matrix $A_{(i,j)}$ is associated to the variable product $q_iq_j$ in~\eqref{eq:BMIdim2}. Therefore, the entries of $A_{(i,j)}$ correspond to the coefficients of $q_iq_j$
in~\eqref{eq:defP11},~\eqref{eq:defP12},~\eqref{eq:defP13P22},~\eqref{eq:defP23} and~\eqref{eq:defP33}.
In the sequel, we define $A_{(1,1)}$.
Observe that $A_{(1,1)}$ is associated to the coefficient $q_1q_1=Q^2_{1,1}$.
Then we collect all the coefficients of $q_1q_1$ in
in~\eqref{eq:defP11},~\eqref{eq:defP12},~\eqref{eq:defP13P22},~\eqref{eq:defP23} and~\eqref{eq:defP33}.
The non-dependent relations, that is $P_{1,1},P_{1,2},P_{2,3},P_{3,3}$ are naturally implemented in $E_{(1,1)}, E_{(1,2)},E_{(2,3)}, E_{(3,3)}$, while the dependent relation is implemented in the diagonal of $E_{(4,4)}$ and $E_{(5,5)}$.
We remark that in order to implement the equality~\eqref{eq:defP13P22} for a dependent relation, we need to implement two inequalities $2P_{1,3} + P_{2,2}-B_{1,2}^2(1-p)\geq 0$ and $2P_{1,3} + P_{2,2}-B_{1,2}^2(1-p)\leq 0$.
Hence, we have

\begin{equation}
\begin{split}
A_{(1,1)}:&=(-2A_{1,1}+(1-p)B^2_{1,1})E_{(1,1)}\\
&\quad+(-2A_{1,2}+B_{1,1}B_{1,2}(1-p))E_{(1,2)}\\
&\quad+(B_{1,2}^2(1-p))(E_{(4,4)}-E_{(5,5)})\\
&\quad +(0)E_{(2,3)}\\
&\quad +(0)E_{(3,3)}.
\end{split}
\end{equation}

Similarly, the matrices
$A_{(1,2)}$, $A_{(1,3)}$, $A_{(2,2)}$, $A_{(2,3)}$ and $A_{(3,3)}$ are defined as follows
\begin{equation}
\begin{split}
A_{(1,2)}:&=(2B_{1,1}B_{2,1}(1-p)-2A_{2,1})E_{(1,1)}\\
&\quad+(-3A_{1,1}-A_{2,2}-B^2_{1,1}-B_{1,1}B_{2,2}-B_{1,2}B_{2,1}\\
&\qquad\quad - B_{1,2}B_{2,1}(p-2)-B_{1,1}(B_{1,1} + B_{2,2})(p -2))E_{(1,2)}\\
&\quad+(-6A_{1,2}+4B_{1,1}B_{1,2}(1-p)+2B_{1,2}B_{2,2}(1-p))(E_{(4,4)}-E_{(5,5)})\\
&\quad +(B_{1,2}^2(1-p))E_{(2,3)}\\
&\quad +(0)E_{(3,3)},
\end{split}
\end{equation}

\begin{equation}
\begin{split}
A_{(1,3)}:&=(-B^2_{2,1})E_{(1,1)}\\
&\quad+(-A_{2,1}-B_{2,1}B_{2,2}-B_{1,1}B_{2,1}(p-2))E_{(1,2)}\\
&\quad+(-2A_{1,1} - 2A_{2,2} - B_{1,1}^2 - B_{2,2}^2\\
&\qquad \quad
-2B_{1,1}B_{2,2}(p - 2)
- 2B_{1,2}B_{2,1}(p - 2))(E_{(4,4)}-E_{(5,5)})\\
&\quad +(-A_{1,2}-B_{1,1}B_{1,2}-B_{1,2}B_{2,2}(p - 2))E_{(2,3)}\\
&\quad +(-B_{1,2}^2)E_{(3,3)},
\end{split}
\end{equation}

\begin{equation}
\begin{split}
A_{(2,2)}:&=(B^2_{2,1}(2-p))E_{(1,1)}\\
&\quad+(-2A_{2,1}-2B_{1,1}B_{2,1}-B_{2,1}(B_{1,1} + B_{2,2})(p - 2))E_{(1,2)}\\
&\quad+(-4A_{1,1} - 4A_{2,2} + (2-p)(B_{1,1} + B_{2,2})^2\\
&\qquad \quad
- 4B_{1,1}B_{2,2} -2p B_{1,2}B_{2,1})(E_{(4,4)}-E_{(5,5)})\\
&\quad +(-2A_{1,2}-2B_{1,2}B_{2,2}-B_{1,2}(B_{1,1} + B_{2,2})(p-2))E_{(2,3)}\\
&\quad +(B_{1,2}^2(2-p))E_{(3,3)},
\end{split}
\end{equation}

\begin{equation}
\begin{split}
A_{(2,3)}:&=(0)E_{(1,1)}\\
&\quad+(B^2_{2,1}(1 - p))E_{(1,2)}\\
&\quad+(-6A_{2,1}+2B_{1,1}B_{2,1}(1-p) + 4B_{2,1}B_{2,2}(1-p))(E_{(4,4)}-E_{(5,5)})\\
&\quad +(-A_{1,1}-3A_{2,2}-B_{2,2}^2-B_{1,1}B_{2,2}\\
&\qquad\quad
+B_{1,2}B_{2,1}(1-p)+B_{2,2}(B_{1,1} + B_{2,2})(2-p))E_{(2,3)}\\
&\quad +(2B_{1,2}B_{2,2}(1-p)-2A_{1,2})E_{(3,3)},
\end{split}
\end{equation}
and finally,
\begin{equation}
\begin{split}
A_{(3,3)}:&=(0)E_{(1,1)}\\
&\quad+(0)E_{(1,2)}\\
&\quad+(-A_{2,1} + B_{2,1}B_{2,2}(1-p))(E_{(4,4)}-E_{(5,5)})\\
&\quad +(-A_{2,1} + B_{2,1}B_{2,2}(1-p))E_{(2,3)}\\
&\quad +(-2A_{2,2}+B_{2,2}^2(1-p))E_{(3,3)}.
\end{split}
\end{equation}

In what follows, we construct the matrices $B_{(1)}$, $B_{(2)}$, $B_{(3)}$, $B_{(4)}$, $B_{(5)}$ and $B_{(6)}$, that is, the
implementation of the matrix $Q$.
Recall the definition of the variables  $q_1$, $q_2$ and $q_3$ given
in~\eqref{defn:qs}. In the sequel, we implement the variables $q_1$, $q_2$ and $q_3$ in the following positions
\[
B_{(1)}:=E_{(6,6)},\quad B_{(2)}:=E_{(6,7)}\quad \textrm{ and }\quad
B_{(3)}:=E_{(7,7)}.
\]
Now, recall~\eqref{defn:otqs}, that is, the definition of the variables $q_4$ and $q_5$ corresponding to the dependent relation~\eqref{eq:2p}. Since a relation of real numbers of the type $u=v$ is implemented as $u-v\geq 0$ and $v-u\geq 0$,
we implement them as follows
\[
B_{(4)}:=E_{(1,3)}+2(E_{(5,5)}-E_{(4,4)})\quad \textrm{ and }\quad
B_{(5)}:=E_{(2,2)}+(E_{(5,5)}-E_{(4,4)}).
\]
The remainder variable $q_6=c$ given in~\eqref{defn:qs} is implemented in $B_6$ as follows
\[
B_{(6)}:=-E_{(1,1)}-2(E_{(4,4)}-E_{(5,5)})-E_{(3,3)}+E_{(8,8)}.
\]

Finally, we construct the matrix $C$.
Recall that $I_2$ denotes the identity in $\textsf{M}_2$.
In the matrix $C$, the $\varepsilon$-condition on the symmetric matrix $Q-\varepsilon I_2$ being non-negative definite and $c-\varepsilon\geq 0$ are implemented as follows
\[
C:=-\varepsilon(E_{(6,6)}+E_{(7,7)}+E_{(8,8)}).
\]

\section{\textbf{Examples and simulations}}\label{sec:examples}\hfill

\noindent
In this section, we exemplify our method for specific models from the literature in dimension two and three. More precisely, in each example, we provide numerical expressions for a matrix $Q$ defining a Lyapunov function for a particular linear system and we also specify the constants  $p$ and
$\varepsilon$ used in the corresponding BMI problem~\eqref{eq:BMIP} or the corresponding LMI problem~\eqref{LMIeq}.
The variable $c$ might be useful in the future, if an objective function is to be minimized. However, in all our examples only the feasibility is of interest and therefore  $c\ge \varepsilon$ translates to $c=\varepsilon$ in all our solutions.
Recall, that we solve  BMI problems with the heuristic from~\cite{KZM2018BMI1,KZM2018BMI2}, which iteratively solves a series of auxiliary  LMI problems.
For setting up LMI problems in Matlab there are several possibilities, e.g.~CVX~\cite{cvx},
YALMIP~\cite{yalmip}, or the LMI Matlab toolbox.  Further, several different solvers can be used, e.g.~MOSEK~\cite{mosek}, SeDuMI~\cite{S98guide} or SDPT3~\cite{sdpt3}.  In our examples we use CVX with SDPT3.

\subsection{\textbf{Dimension two}}\label{sec:dimtwo}
\hfill

\noindent
In this subsection, we study two dimensional systems such as
random linear oscillators, linear satellite dynamics, two inertia system and geometric Brownian motion on the circle.

\subsubsection{\textbf{Random linear oscillator}}\label{sec:rlo}
\hfill

\noindent
Random (noisy) linear oscillators are perhaps the simplest toy models used as a prototype for different phenomena in nature,
see~\cite{GittermanMulti,Gitterman}.
In the sequel, we analyze the stability of a random linear oscillator with multiplicative noise.
Consider the following second order deterministic differential equation (deterministic oscillator)
\[
\ddot{x}(t)+\gamma \dot{x}(t)+
\kappa x(t)=0 \quad \textrm{ with initial data }\quad x(0)=Q_0\quad \textrm{ and }\quad \dot{x}(0)=P_0,
\]
where $\gamma$ and $\kappa$ are parameters that represent the damping and the intensity of the force, respectively.
Assume that the fluctuations of the parameters $\gamma$ and $\kappa$ are modeled by two independent Brownian motions $W_1$ and $W_2$, that is by
$\gamma+\sigma_1 \ud W_1$ and $\kappa+\sigma_2 \ud W_2$, where $\sigma_1$ and $\sigma_2$ are non-zero real numbers.
Using the position-momentum variables $((Q(t;x),P(t;x))^*)_{t\geq 0}$, the latter yields the linear SDE
\begin{equation}
\begin{split}
\ud
\begin{pmatrix}
Q(t;x)\\
P(t;x)
\end{pmatrix}&=
A
\begin{pmatrix}
Q(t;x)\\
P(t;x)
\end{pmatrix}\ud t\\
&\qquad+
B_1
\begin{pmatrix}
Q(t;x)\\
P(t;x)
\end{pmatrix} \ud W_1(t)
+B_2
\begin{pmatrix}
Q(t;x)\\
P(t;x)
\end{pmatrix} \ud W_2(t), \quad t\geq 0
\end{split}
\end{equation}
with initial condition $Q(0;x)=Q_0$ and $P(0;x)=P_0$,
where $x=(Q_0,P_0)^*$,
\[
A:=
\begin{pmatrix}
0 & 1 \\
-\kappa & -\gamma
\end{pmatrix},
\quad
B_1:=
\begin{pmatrix}
0 & 0\\
0 & -\sigma_1
\end{pmatrix}
\quad \textrm{ and }\quad
B_2:=\begin{pmatrix}
0 & 0\\
-\sigma_2 & 0
\end{pmatrix}.
\]
For the values
$\kappa=1$, $\gamma=0.2$, $\sigma_1=0.3$, and $\sigma_2=0.5$
 and with $p=2$, $\varepsilon=0.01$ the LMI
 problem~\eqref{LMIeq} has the solution
\begin{equation}
Q=
\begin{pmatrix}
8.593474800605289  & 1.240664128315077\\
1.240664128315077 &  8.369114515122694
\end{pmatrix},
\end{equation}
which asserts exponential mean-square stability of the origin.
For the values
$\kappa=1$, $\gamma=0.1$, $\sigma_1=0.3$, and $\sigma_2=0.5$, i.e.~the damping has been reduced by factor two,  and with $p=2$,  the LMI
problem~\eqref{LMIeq} does not have a solution and the null solution is not exponentially mean-square stable, see Remark~\ref{rem:lmip2}.  However, with $p=0.1$ and  $\varepsilon=0.01$ the BMI
problem~\eqref{eq:BMIP} has
a solution
\begin{equation}
Q=
\begin{pmatrix}
0.809364631343529  & 0.100867430488615\\
0.100867430488615 &  0.773229253083311
\end{pmatrix},
\end{equation}
and the origin is $0.1$-ES stable, and thus GASiP.

\subsubsection{\textbf{Satellite dynamics}}\label{sec:satelite}
\hfill

\noindent
In this subsection, we analyze the stability of linearized satellite dynamics, see~\cite{Sagirow1970} and Chapter~5 in~\cite{Burragethesis}.
It is a simplified version of the non-linear satellite dynamics
which arise in the study of the influence of a rapidly fluctuating density of the Earth's atmosphere on the motion of a satellite in a circular orbit~\cite{Sagirow1970}.
To be more precise, the non-linear satellite dynamics are modeled by the second order deterministic differential equation
\[
\ddot{x}(t)+\gamma \dot{x}(t)+
\kappa x(t)=\alpha \sin(2x) \quad \textrm{ with initial data }\quad x(0)=Q_0\quad \textrm{ and }\quad \dot{x}(0)=P_0,
\]
where $\gamma$, $\kappa$ and $\alpha$ are parameters.
Analogously to the Subsection~\ref{sec:rlo} we assume that the fluctuations of the parameters $\gamma$ and $\kappa$ are modeled by two independent Brownian motions $W_1$ and $W_2$, that is by
$\gamma+\sigma_1 \ud W_1$ and $\kappa+\sigma_2 \ud W_2$ where $\sigma_1$ and $\sigma_2$ are non-zero real numbers.
Note that in the position-momentum variables $((Q(t;x),P(t;x))^*)_{t\geq 0}$, the point $(0,0)^*$ is a fixed point of the non-linear dynamics. Therefore, we consider a linearization around the fixed point $(0,0)^*$ and obtain the following linear SDE in the linearized variables
$((Q^{\textrm{Lin}}(t;x),P^{\textrm{Lin}}(t;x))^*)_{t\geq 0}$
\begin{equation}
\begin{split}
\ud
\begin{pmatrix}
Q^{\textrm{Lin}}(t;x)\\
P^{\textrm{Lin}}(t;x)
\end{pmatrix}&=
A
\begin{pmatrix}
Q^{\textrm{Lin}}(t;x)\\
P^{\textrm{Lin}}(t;x)
\end{pmatrix}\ud t\\
&\qquad+
B_1
\begin{pmatrix}
Q^{\textrm{Lin}}(t;x)\\
P^{\textrm{Lin}}(t;x)
\end{pmatrix} \ud W_1(t)
+B_2
\begin{pmatrix}
Q^{\textrm{Lin}}(t;x)\\
P^{\textrm{Lin}}(t;x)
\end{pmatrix} \ud W_2(t), \quad t\geq 0
\end{split}
\end{equation}
with initial condition $Q^{\textrm{Lin}}(0;x)=Q_0$ and $P^{\textrm{Lin}}(0;x)=P_0$,
where $x=(Q_0,P_0)^*$,
\[
A:=
\begin{pmatrix}
0 & 1 \\
-\kappa+2\alpha & -\gamma
\end{pmatrix},
\quad
B_1:=
\begin{pmatrix}
0 & 0\\
0 & -\sigma_1
\end{pmatrix}
\quad \textrm{ and }\quad
B_2:=\begin{pmatrix}
0 & 0\\
-\sigma_2 & 0
\end{pmatrix}.
\]
For the values
$\kappa=1$, $\gamma=0.2$, $\alpha=0.2$, $\sigma_1=0.3$ and $\sigma_2=0.5$ we obtain a solution to the BMI problem~\eqref{eq:BMIP} with $p=0.1$ and $\varepsilon=0.01$;
\begin{equation}
Q=
\begin{pmatrix}
0.470134826272870 &  0.153025027071006\\
0.153025027071006 &  0.675335825647526
\end{pmatrix}.
\end{equation}
With $p=2$  the LMI problem~\eqref{LMIeq} does not have a solution for these parameter values, hence, the origin is $0.1$-ES and GASiP, but not exponentially mean-square stable.
Note that LMIs are exactly solvable and that, unless numerical troubles are encountered, the solver either finds a feasible solution or reports that there do not exist any feasible solutions.

\subsubsection{\textbf{Two-inertia systems}}\label{sec:inertia}
\hfill

\noindent
In this subsection, we analyze a two-inertia linear system, which is a mechanical system driven by electrical motors, see~\cite{Hutwo}. It is described as the solution of
the linear SDE
\begin{equation}
\begin{split}
\ud
\begin{pmatrix}
\theta_m(t;\theta)\\
\theta_L(t;\theta)
\end{pmatrix}&=
A
\begin{pmatrix}
\theta_m(t;\theta)\\
\theta_L(t;\theta)
\end{pmatrix}\ud t\\
&\qquad+
B_1
\begin{pmatrix}
\theta_m(t;\theta)\\
\theta_L(t;\theta)
\end{pmatrix} \ud W_1(t)
+B_2
\begin{pmatrix}
\theta_m(t;\theta)\\
\theta_L(t;\theta)
\end{pmatrix} \ud W_2(t), \quad t\geq 0
\end{split}
\end{equation}
with initial condition $\theta_m(0;0)=\theta_L(0;0)=\theta$, where $\rho$ is a parameter,
\[
A:=
\begin{pmatrix}
-\frac{\rho}{2} & \frac{\rho}{2} \\
\frac{\rho}{2} & -\frac{\rho}{2}
\end{pmatrix},
\quad
B_1:=
\begin{pmatrix}
\sigma_1 & 0\\
0 & 0
\end{pmatrix}
\quad \textrm{ and }\quad
B_2:=\begin{pmatrix}
0 & 0\\
0 & \sigma_2
\end{pmatrix}.
\]
For the values
$\rho=2$, $\sigma_1=0.3$, and $\sigma_2=0.5$ we obtain with $\varepsilon=0.01$  $p$-ES stability with $p=0.1$; the matrix in the solution is
\begin{equation}
Q=
\begin{pmatrix}
0.456131392079160 &  0.133662106581867\\
 0.133662106581867  & 0.448403918451280
\end{pmatrix}.
\end{equation}
For $p=2$ the LMI problem~\eqref{LMIeq} does not yield a solution and therefore the origin is not exponentially mean-square stable, although it is GASiP.

\subsubsection{\textbf{Diagonal noise system}}\label{sec:diagonal}
\hfill

\noindent
In this subsection, we analyze independent diagonal noises for a test system from~
Section~3.1 in~\cite{Senosiain}. More precisely, the unique strong solution of the following linear SDE in $\mathbb{R}^2$
\begin{equation}\label{eq:idnoi}
\ud X(t;x) = AX(t;x)\ud t+
\sum_{j=1}^\ell B_j X(t;x)\ud W_j(t),\quad t\geq 0,
\end{equation}
with initial state $X(0;x)=x\in \mathbb{R}^2$,
parameters $\lambda,b,\sigma_1,\ldots,\sigma_\ell\in \mathbb{R}$,
and
\[
A:=
\begin{pmatrix}
\lambda & b\\
0 & \lambda
\end{pmatrix}\quad \textrm{ and }\quad
B_j:=\sigma_j \begin{pmatrix}
1 & 0\\
0 & 1
\end{pmatrix},\quad j=1,\ldots,\ell.
\]
Proposition~2 in~\cite{Senosiain} yields
that~\eqref{eq:idnoi} is exponentially mean-square stable, if and only if,  \[
2\lambda+\sum_{j=1}^\ell \sigma_j^2<0.
\]

In the sequel, we set $\lambda=-1$, $b=2$, $\sigma_j=2$ for $1\leq j\leq \ell$, $p=0.1$ and $\varepsilon=0.001$.
First, we set $\ell=1$ and obtain a solution to the BMI problem~\eqref{eq:BMIP} with
\begin{equation}
Q=
\begin{pmatrix}
0.473458820225827 &  0.178684084077058\\
   0.178684084077058  & 0.185153654570083
\end{pmatrix},
\end{equation}
although $2\lambda+\sigma_1^2=2>0$.  Hence, the origin is $0.1-$ES, and therefore GASiP. However, it is not exponentially mean-square stable.

Now for $\ell=2$, $\ell=3$, and $\ell=6$ we obtain a solution to the BMI problem~\eqref{eq:BMIP} with
\begin{align*}
Q&=
\begin{pmatrix}
 0.454559027917115 &  0.126148161036448\\
  0.126148161036448 &  0.107293494631372
\end{pmatrix}\quad \textrm{ for } \quad \ell=2,\\
Q&=
\begin{pmatrix}
0.436005113062966 &  0.096327529945338\\
 0.096327529945338  & 0.073678666871984
\end{pmatrix}\quad \textrm{ for } \quad \ell=3\quad \textrm{ and }\\
Q&=
\begin{pmatrix}
0.415228565039829 &  0.046662438957609\\
 0.046662438957609 &  0.042176845241893
\end{pmatrix}\quad \textrm{ for } \quad \ell=6,
\end{align*}
although $2\lambda+\sum_{j=1}^\ell \sigma_j^2=6,10,22$ for $\ell=2,3,6$, respectively.  Hence, in all cases, the origin is GASiP but it is not exponentially mean-square stable.  The preceding example highlights  the conservativeness of exponential mean-square stability if the intention is to verify
GASiP.

\subsubsection{\textbf{Off-diagonal noise systems}}\label{sec:offdiagonal}

\hfill

\noindent
In this section, we  investigate the stability of the null solution for the SDEs
\begin{equation}\label{eq:offdia}
\ud X(t;x) = AX(t;x)\ud t+ B_i X(t;x)\ud W(t),\quad t\geq 0,\quad i=1,2,
\end{equation}
with initial state $X(0;x)=x\in \mathbb{R}^2$, parameters $\lambda,b,\sigma_1,\sigma_2 \in \mathbb{R}$, and
\[
A:=
\begin{pmatrix}
\lambda & b\\
0 & \lambda
\end{pmatrix},\quad
B_1:=\sigma_1 \begin{pmatrix}
0 & -1\\
1 & 0
\end{pmatrix}, \quad\textrm{ and }\quad B_2:=\sigma_2 \begin{pmatrix}
0 & 0\\
1 & 0
\end{pmatrix}.
\]
With $i=1$, $\lambda=-1$, $b=2$, $\sigma_1=1.1$, $p=0.1$ and $\varepsilon=0.001$
we obtain a solution to the BMI problem~\eqref{eq:BMIP} with
\begin{equation}
Q=
\begin{pmatrix}
   0.297895991144839  & 0.147907009129502 \\
   0.147907009129502 &  0.579136945567880
\end{pmatrix},
\end{equation}
however, with $p=2$ the LMI problem~\eqref{LMIeq} does not yield a solution.  Hence, the origin is $0.1$-ES stable, and thus GASiP, although it is not exponentially mean-square stable for these parameter values.
With $\sigma_1=1$, $p=2$ and all the other parameters unchanged, we get the solution
\begin{equation}
Q=
\begin{pmatrix}
   8.138668959061926 &  5.037463351312330 \\
   5.037463351312330 & 15.354367614159404
\end{pmatrix}
\end{equation}
to the LMI problem~\eqref{LMIeq} and the origin is exponentially mean-square stable.
While for $i=2$, $\lambda=-1$, $b=2$, $\sigma_2=1.7$, $p=0.1$ and $\varepsilon=0.001$
we obtain a solution to the BMI problem~\eqref{eq:BMIP} with
\begin{equation}
Q=
\begin{pmatrix}
   0.407005635565675  & 0.131672012443728 \\
   0.131672012443728 &  0.361655336346576
\end{pmatrix},
\end{equation}
but with $p=2$ the LMI problem~\eqref{LMIeq} does not have a solution.  Hence, again, the origin is $0.1$-ES stable, and thus GASiP, although it is not exponentially mean-square stable.
For $\sigma_2=0.9999997$, $p=2$ and all the other parameters untouched, we obtain a solution
\begin{equation}
Q=
\begin{pmatrix}
     82.1795694954346  & 82.1795204431116 \\
   82.1795204431116  & 164.3591391900936
\end{pmatrix}
\end{equation}
to the LMI problem~\eqref{LMIeq} and the origin is exponentially mean-square stable.  Finally, for $\sigma_2=0.9999998$ we do not get a solution to the LMI problem~\eqref{LMIeq}.

\subsection{\textbf{Dimension three}}\label{sec:dimthree}
\hfill

\noindent
In this subsection, we study three dimensional systems of  cancer self-remission modeling and smoking modeling.

\subsubsection{\textbf{Cancer self-remission and tumor stability}}\label{sec:cancertumor}\hfill

\noindent
We consider the linearization around the unique positive and stable equilibrium of the non-linear cancer self-remission and tumor stability model introduced in~\cite{Sarkar}.
Here, positive means that all components of the vector are positive.
Let $(X(t;x))_{t\geq 0}$ be the solution of the linear SDE
\begin{equation}
\ud X(t;x) = AX(t;x)\ud t+
B_1 X(t;x)\ud W_1(t)+B_2 X(t;x)\ud W_2(t)+B_3 X(t;x)\ud W_3(t),\quad t\geq 0,
\end{equation}
where
\begin{equation}\label{eq:defAvieja}
\begin{split}
&A:=
\begin{pmatrix}
\delta_1 & \delta_3 & 0\\
0 & 0 & \frac{\beta}{\alpha}\delta_2\\
0 & -d_1 & -\frac{sd_1}{\beta k_2}
\end{pmatrix},
\quad
B_1:=\begin{pmatrix}
\sigma_1 & 0 & 0\\
0 & 0 & 0\\
0 & 0 & 0
\end{pmatrix},\quad
B_2:=\begin{pmatrix}
0 & 0 & 0\\
0 & \sigma_2 & 0\\
0 & 0 & 0
\end{pmatrix}\\
&\textrm{and }\quad
B_3:= \begin{pmatrix}
0 & 0 & 0\\
0 & 0 & 0\\
0 & 0 & \sigma_3.
\end{pmatrix}
\end{split}
\end{equation}
Here  $r$, $q$, $k_1$, $k_2$, $\alpha$, $\beta$, $d_1$, $d_2$, $s$ are positive constants, $\sigma_1$, $\sigma_2$ and $\sigma_3$ are non-negative constants and
\begin{equation}\label{eq:defdeltavieja}
\begin{split}
\delta_1:&=-\sqrt{\left[\frac{\alpha s}{\beta}\left(1-\frac{d_1}{\beta k_2}\right)-\frac{\alpha d_2}{\beta}-r\right]^2+\frac{4rq}{k_1}}<0,\\
\delta_2:&=\alpha N^\dagger>0,\ \ \text{with}\ \ N^\dagger:= \frac{ s}{\beta}\left(1-\frac{d_1}{\beta k_2}\right)-\frac{ d_2}{\beta},\\
\delta_3:&=-\alpha \left(\frac{-(\alpha N^\dagger-r)+\sqrt{(\alpha N^\dagger-r)^2+\frac{4rq}{k_1}}}{2\frac{r}{k_1}}\right)<0.
\end{split}
\end{equation}
For a full interpretation and restrictions of the variables and parameters, see~Appendix~\ref{ap:sarkar}.

In the sequel, we study the example given in~\cite{Sarkar}. More precisely,
with the parameters $\alpha=0.3$, $\beta=0.1$, $q=10$, $r=0.9$, $s=0.8$, $k_1=0.8$, $k_2=0.7$, $d_1=0.02$, $d_2=0.03$, as in Table~1, p.~71 of~\cite{Sarkar}, and $\sigma_1=3.67$, $\sigma_2=0.13$, $\sigma_3=1.63$
used in Fig.~1, p.~76 of~\cite{Sarkar},
 the sufficient conditions given in Theorem~\ref{Cneccond} in Appendix~\ref{ap:sarkar} for the exponential mean-square stability of the null solution cannot be fulfilled for
$\sigma_3=1.63$. Indeed, since $\delta_2>0$ and $w_3+w_4>0$, we obtain the following contradiction
\begin{equation}\label{eq:inesig}
1.63^2=\sigma^2_3=\frac{-2\frac{\beta}{\alpha}\delta_2 w_4}{w_3+w_4} +2\frac{sd_1}{\beta k_2} \le 2\frac{sd_1}{\beta k_2} =\frac{2}{5} < 1.63^2.
\end{equation}
This is in contrast to the assertion that the null solution is exponentially mean-square stable
in~\cite{Sarkar}.
Set $p=2$ and recall that the corresponding LMI
problem~\eqref{LMIeq} obtained from the BMI
problem~\eqref{eq:BMIP} for the above parameters are equivalent.
One can verify that the  LMI problem~\eqref{LMIeq} does not have a solution and the null solution is not exponentially mean-square stable for these 
 parameters values.
However, for $\sigma_3=0.4\cdot 1.63=0.625$ we have that $2/5=0.4>\sigma_3^2=0.390625$ and now the  LMI problem~\eqref{LMIeq} has the solution
 \begin{equation}
 \label{Qmatcancer}
Q=
\begin{pmatrix}
   0.058076730721533 & -0.006898010784651 & -0.000513990964685 \\
  -0.006898010784651 &  4.459961467243316 &  2.361679166055028 \\
  -0.000513990964685 &  2.361679166055028 & 95.032770564673839
\end{pmatrix}
\end{equation}
for $\varepsilon=0.01$.
Now, we increase the intensity of the noise by $5\%$, i.e., set $\sigma_1=1.05\cdot3.67$, $\sigma_2=1.05\cdot 0.13$, and $\sigma_3=1.05\cdot0.625$. Then the sufficient conditions given in Theorem~\ref{Cneccond} in Appendix~\ref{ap:sarkar} are neither fulfilled for $\sigma_1$, because
$\sigma_1^2>-2\delta_1$, nor for $\sigma_3$ due to a  similar  computation as  in~\eqref{eq:inesig}.
For these parameters the corresponding LMI problem does also not have a solution and the null solution is not exponentially mean-square stable.  If we, however, use the matrix $Q$ given in~\eqref{Qmatcancer} and solve the corresponding LMI problem obtained from the BMI problem by assigning values from the
matrix $Q$ to the corresponding variables, we then obtain a solution
with $p=0.01$ and $\varepsilon=10^{-5}$; for a discussion of such LMI problems
see~\cite{HafsteinGudmundsson}. 
However, if we use the matrix $Q$ given in~\eqref{Qmatcancer} and assign its values to the corresponding variables in the BMI problem \eqref{eq:BMIP}, we obtain an LMI problem that has a feasible solution with $p=0.01$ and $\varepsilon=10^{-5}$.
To see how the BMI becomes an LMI by assigning these values, consider the BMI in the form \eqref{eq:defJmatrix}, with $m=n(n+1)/2$ and $K>m+1$, as written in the proof of Theorem  \ref{th:stability}.  By fixing the values of the variables $q_1,q_2,\ldots,q_m$ to the entries in the upper-triangular part of $Q$, row-by-row, the BMI \eqref{eq:defJmatrix} becomes an LMI because all quadratic terms are fixed constants.
Similar LMI problems are discussed in \cite{HafsteinGudmundsson}.
Thus, the null solution is indeed $0.01$-exponentially stable and therefore GASiP.  Nevertheless, using the heuristic from~\cite{KZM2018BMI1,KZM2018BMI2} to compute a Lyapunov function directly using the BMI problem remained unsuccessful.
We presume that this is because that the needed matrices $Q$ for a feasible solution are rather ill-conditioned (big condition number), as the $Q$ given
in~\eqref{Qmatcancer} has eigenvalues $\lambda_1\approx 0.0581$,  $\lambda_2\approx 4.3984$, $\lambda_3=  95.0943$, i.e., its condition number $\lambda_3/\lambda_1$ is approximately $1637$.

\subsubsection{\textbf{Stochastic stability of a mathematical model of smoking}}\label{sec:smookingstable}\hfill

\noindent
In this subsection, we analyze the linearization around the unique endemic equilibrium state of the non-linear
 stochastic dynamics for a mathematical model of smoking given in Equation~(3) of~\cite{Lahrouz}.

Let $E_*=(P_*,S_*,Q_{T,*})$ denote the
endemic equilibrium state (smoking-present equilibrium) of the Model~(2) of~\cite{Lahrouz}.
More precisely,
\begin{equation}
\begin{split}
P_*:=\frac{1}{R},\quad
S_*:=\frac{\mu}{\beta}(R-1)\quad \textrm{ and }\quad
Q_{T,*}:=\frac{\gamma(1-\sigma)}{\mu+\alpha}S_*
\end{split}
\end{equation}
where the parameters $\alpha$, $\beta$, $\gamma$, $\mu$ are positive,  $\sigma\in (0,1)$ and $R>1$ is defined as follows
\begin{equation}
R:=\frac{\beta(\mu+\alpha)}{\mu(\mu+\alpha)+\gamma(\sigma\alpha+\mu)}.
\end{equation}
For a full description of the parameters and variables we refer to~\cite{Lahrouz, Sharomi}.
As in Section~3.3 of~\cite{Lahrouz}, we now consider the linearization around $E_*$ for
Equation~(3) in~\cite{Lahrouz}. More precisely, we consider
the unique strong solution to the following linear SDE:
\begin{equation}
\ud X(t;x) = AX(t;x)\ud t+
B_1 X(t;x)\ud W_1(t)+B_2 X(t;x)\ud W_2(t)+B_3 X(t;x)\ud W_3(t),\quad t\geq 0,
\end{equation}
where
\begin{equation}
\begin{split}
&A:=
\begin{pmatrix}
-(\mu+\beta S_*) & -\beta P_* & 0\\
\beta S_* & -(\mu+\gamma-\beta P_*) & \alpha\\
0 & \gamma(1-\sigma)& -(\mu+\alpha)
\end{pmatrix},\quad
B_1:=\begin{pmatrix}
\sigma_1 P_* & 0 & 0\\
0 & 0 & 0\\
0 & 0 & 0
\end{pmatrix},\\
&
B_2:=\begin{pmatrix}
0 & 0 & 0\\
0 & \sigma_2 S_* & 0\\
0 & 0 & 0
\end{pmatrix}\quad \textrm{ and }\quad
B_3:= \begin{pmatrix}
0 & 0 & 0\\
0 & 0 & 0\\
0 & 0 & \sigma_3 Q_{T,*}
\end{pmatrix}.
\end{split}
\end{equation}
with positive constants $\sigma_1$, $\sigma_2$ and $\sigma_3$.
The components of $X(t;x)$ are interpreted as follows: the first component
$X_1(t;x)$ denotes the proportion of potential smokers at time $t$, the second component
$X_2(t;x)$ denotes the proportion of smokers at time $t$ and  the third component $X_3(t;x)$ denotes the proportion of smokers who temporarily quit
smoking at time $t$.

In Theorem~3.4 of~\cite{Lahrouz},  sufficient conditions in the parameters of the model are derived for the exponential mean-square stability ($p=2$) of its solution $(X(t;x))_{t\geq 0}$, i.e.,~the exponential mean-square stability of the equilibrium $E_*$. Such conditions take the form of upper bounds on
the noise intensity parameters $\sigma_1,\sigma_2,\sigma_3$. In addition, an analytical formula for a Lyapunov function is provided in the proof of Theorem~3.4 in~\cite{Lahrouz}.

In what follows, we fix the following parameters:
 $\alpha=0.3$, $\beta=2$, $\gamma=1$, $\mu=1$ and $\sigma=0.8$.
With the parameters $\sigma_1=1, \sigma_2=12,\sigma_3=500$ the conditions of Theorem~3.4 in~\cite{Lahrouz} are fulfilled and with $p=2$ the LMI problem~\eqref{LMIeq}  with $\varepsilon=0.01$ gives
\begin{equation}
Q=
\begin{pmatrix}
0.193270606602671 &  0.575166417171658 &  0.109139637875147 \\
   0.575166417171658 &  21.979601279704919  & 4.898808646630821 \\
   0.109139637875147 &  4.898808646630821 &  7.695985872174183
\end{pmatrix}.
\end{equation}
With the parameters $\sigma_1=1, \sigma_2=23,\sigma_3=700$  the sufficient conditions given in Theorem~3.4 of~\cite{Lahrouz} are not satisfied for $\sigma_2$ and $\sigma_3$, but the LMI
problem~\eqref{LMIeq} has a solution and thus verifies exponential mean-square stability  with $\varepsilon=0.01$ and delivers
\begin{equation}
Q=
\begin{pmatrix}
0.073696787670233  & 0.820536691302952  & 0.223027127168492 \\
   0.820536691302952  &38.990165056725857 &  9.271874680978858 \\
   0.223027127168492 &  9.271874680978858 &  8.335304714592960
\end{pmatrix}.
\end{equation}
With the parameters $\sigma_1=1.15, \sigma_2=26.45,\sigma_3=805$, i.e.,~all $\sigma_1$, $\sigma_2$ and $\sigma_3$ are $15\%$ larger, the  LMI
problem~\eqref{LMIeq} does not have a solution and the null solution is not exponentially mean-square stable.

However, using the heuristic
from~\cite{KZM2018BMI1,KZM2018BMI2} to solve the BMI problem~\eqref{eq:BMIP} we obtained a solution even with $\sigma_1=1.2, \sigma_2=27.6,\sigma_3=840$, i.e.~all $\sigma_1$, $\sigma_2$, and $\sigma_3$ are $20\%$ larger, setting $p=0.1$ and $\varepsilon=10^{-5}$;
 \begin{equation}
Q=
\begin{pmatrix}
   0.014277198281505 &  0.001654320493336 & -0.022100683099102 \\
   0.001654320493336 &  0.562180335603617 &  0.117951796401118 \\
  -0.022100683099102 &  0.117951796401118 &  0.537489973802403
\end{pmatrix}.
\end{equation}
Hence, for these parameters,  the endemic equilibrium $E_*$ is $0.1$-exponentially stable in probability, and thus GASiP, although it is not exponentially mean-square stable.

\section{\textbf{Conclusions}}
\label{sec:conclusions}\hfill

\noindent
We have shown that for linear stochastic differential equations of the form~\eqref{eq:modelito}  a Lyapunov function of the form
\[
V(x)=\|x\|_Q^p:=\langle x,Qx\rangle^{\frac{p}{2}},\quad x\in \mathbb{R}^n,
\]
where $p>0$ and $Q$ is a symmetric and positive definite matrix, can be computed using a certain
bilinear matrix inequality (BMI) problem.

A feasible solution to the BMI problem gives the matrix $Q$ and hence a Lyapunov function for the system~\eqref{eq:modelito}, which asserts the $p$-th (moment) exponential stability ($p$-ES) of the null solution, and thus its global asymptotic stability in probability (GASiP).  If $p=2$ the BMI simplifies
to a linear matrix inequality (LMI)
and the exponential mean-square stability ($2$-ES) of the null solution is equivalent to the feasibility of the LMI problem.

We demonstrate for numerous examples from the literature that our novel approach can assert GASiP for the null solution, although the null solution is not exponentially mean-square stable.

We point out that the matrices $A, B_1\ldots, B_{\ell}$ in~\eqref{eq:modelito} are time-independent. As an interesting problem and future work, we would like to study a similar approach for 
time-dependent matrix coefficients, which a priori it does not follows straightforwardly. This will allow us to study the stability of interesting models such as the motion of a helicopter blade, which is a modeled by a multivariate geometric Brownian motion with periodic matrix coefficients
that depend on different parameters such as the velocity of the
helicopter or some geometrical characteristics of the blade, see Chapter~5.1~in~\cite{Talay}.

\appendix
\section{\textbf{Proofs of Lemma~\ref{lem:lyapunov} and Theorem~\ref{th:stability}}}\label{ap:proofmaintheorem}
\hfill

\noindent
In this section, we prove Lemma~\ref{lem:lyapunov} and Theorem~\ref{th:stability}.

\subsection{\textbf{Proof of Lemma~\ref{lem:lyapunov}}}\label{sub:prueba}\hfill

\noindent
\begin{proof}[Proof of Lemma~\ref{lem:lyapunov}]
The proof follows from Lemma~4.2 in~\cite{HafsteinGudmundsson}. For completeness we provide the main steps here.
Assume that~\eqref{eq:Hine} holds true.
The proof that the null solution is GASiP follows
straightforwardly from  Lemma~4.2 in~\cite{HafsteinGudmundsson}.
Indeed,
for the function $V(x)=\|x\|^p_Q$, $x\in \mathbb{R}^n$, a straightforward computation yields
\begin{equation}\label{eq:LVH}
(\mathcal{L}V)(x)=-\frac{p}{2}\|x\|_Q^{p-4}H(x)\quad \textrm{ for all }\quad x\in \mathbb{R}^n
\end{equation}
with
\begin{equation}\label{eq:H}
\begin{split}
H(x)&=-x^* \left( A^* Q+  QA+ \sum_{j=1}^\ell B_j^* Q B_j\right)x \|x\|_Q^2 + \frac{2-p}{4}\sum_{j=1}^\ell \left(x^* (QB_j+B_j^* Q)x\right)^2,
\end{split}
\end{equation}
where $\mathcal{L}V$ is defined in~\eqref{eq:lyapunov}.
Since we are assuming that inequality~\eqref{eq:Hine} holds true,~\eqref{eq:LVH} implies Item~(ii) in Theorem~\ref{th:tool} for a suitable positive constant $c_3$ and hence the null solution is GASiP.
\end{proof}

\subsection{\textbf{Proof of Theorem~\ref{th:stability}}}\label{sec:proofthnew}\hfill

\begin{proof}[Proof of Theorem~\ref{th:stability}]
In the sequel, we show that the
inequality~\eqref{eq:Hine} can be algorithmically rewritten as a BMI with variables $Q$ and $c$.

\noindent
\textbf{Step 1: the order and cardinality of a set of multi-indices.}
We note that $H$ given in~\eqref{cond:H} is an homogeneous polynomial of degree $4$, see Remark~\ref{rem:multihomog}.
Let $\mathcal{I}_n$ be the set of all $n$-dimensional multi-indices of length $4$, that is,
\[
\mathcal{I}_n:=\left\{\alpha\in \mathbb{N}^n_0: |\alpha|:=\sum_{j=1}^{n}\alpha_j=4\right\},
\]
where $\mathbb{N}_0:=\{0,1,2,\ldots\}$.
It is clear that $\mathcal{I}_n$ has finite cardinality.
We  equip $\mathcal{I}_n$ with the reverse lexicographical order and then $\alpha_n(j)$ denotes the $j$-th element of $\mathcal{I}_n$ with respect to such order. For instance, for $n=2$ we obtain
\[
\alpha_2(1)=(4,0),\,\alpha_2(2)=(3,1),\,\alpha_2(3)=(2,2),\,\alpha_2(4)=(1,3),\, \alpha_2(5)=(0,4),
\]
while for $n=3$ we have
\begin{equation}
\begin{split}
&\alpha_3(1)=(4,0,0),\,\alpha_3(2)=(3,1,0),\,\alpha_3(3)=(3,0,1),\,\alpha_3(4)=(2,2,0),\, \alpha_3(5)=(2,1,1),\\
&\alpha_3(6)=(2,0,2),\, \alpha_3(7)=(1,3,0),\,
\alpha_3(8)=(1,2,1),\,
\alpha_3(9)=(1,1,2),\,
\alpha_3(10)=(1,0,3),\,\\
&
\alpha_3(11)=(0,4,0),\,
\alpha_3(12)=(0,3,1),\,
\alpha_3(13)=(0,2,2),\,
\alpha_3(14)=(0,1,3),\,
\alpha_3(15)=(0,0,4).
\end{split}
\end{equation}

We say that two multi-indexes in $\mathcal{I}_n$ are equivalent
if one can be obtained from the other by  a permutation of its components.
We note that this operation defines an equivalence relation and hence a natural partition in equivalence classes.
We observe there are at most five
equivalence classes and we identify such classes according to the numbers that pop up in their components as follows:
\begin{equation}\label{def:class}
\begin{split}
&\mathcal{C}_1\equiv  [4],\,\,
\mathcal{C}_2\equiv [3,1],\,\,
\mathcal{C}_3\equiv [2,2],\,\,\mathcal{C}_4\equiv [2,1,1]\quad \textrm{ and }\quad
\mathcal{C}_5\equiv [1,1,1,1].
\end{split}
\end{equation}
For instance, for $n=2$ we have that
$\alpha_2(1)$ and $\alpha_2(5)$ are of Class 1, $\alpha_2(2)$ and $\alpha_2(4)$ are of Class 2, $\alpha_2(3)$ is of Class 3.
For $n=3$ we obtain four classes and for $n\geq 4$ all the five classes appear.

Let $|\mathcal{I}_n|$ denote the
cardinality of $\mathcal{I}_n$ and
let $|\mathcal{C}_j|$ denote the cardinality of $\mathcal{C}_j$ for each
$j\in \{1,2,3,4,5\}$.
Let $n\in \{2,3,\ldots\}$.
The partition~\eqref{def:class}  yields
\begin{equation}
\begin{split}
|\mathcal{I}_n|&=|\mathcal{C}_1|+|\mathcal{C}_2|+|\mathcal{C}_3|+|\mathcal{C}_4|+|\mathcal{C}_5|\\
&=n+n(n-1)+n\frac{(n-1)}{2}+n\frac{(n-1)(n-2)}{2}+{{n}\choose{4}}\mathds{1}_{\{n\geq 4\}},
\end{split}
\end{equation}
where $\mathds{1}_{\{n\geq 4\}}:=1$ if $n\geq 4$, and $\mathds{1}_{\{n\geq 4\}}=0$ if $n<4$.

Our goal is to find  for an a\,priori fixed $\varepsilon>0$
\begin{itemize}
\item[(a)]
a constant $c\ge \varepsilon$,
\item[(b)] a positive definite symmetric matrix $Q\in \mathsf{M}_n$, $Q \succeq \varepsilon I_n$
\item[(c)] and a symmetric and positive semi-definite matrix $P\in \mathsf{M}_m$
\end{itemize}
such that
\begin{equation}\label{eq:Pzuno}
P_c(x)=z^*Pz \quad \textrm{ for all }\quad z\in \mathbb{R}^m,
\end{equation}
where $m:=n(n+1)/2$ and
\begin{equation}\label{def:zuno}
z:=(x^2_1,x_1x_2,x_1x_3,\ldots, x_1x_{n},x^2_2,x_2x_3,\ldots,x_2x_n,x^2_3, x_3x_4,\ldots,x_3x_n,\ldots,x^2_{n-1}, x_{n-1}x_n, x^2_n)^*.
\end{equation}
We now take four partial derivatives on both sides of~\eqref{eq:Pzuno}. In other words,
\begin{equation}\label{eq:Pzmultione}
\partial^4_{\alpha_n(j)}(P_c(x))=\partial^4_{\alpha_n(j)}(z^*Pz)\quad \textrm{ for all }\quad j\in \{1,2,\ldots,|\mathcal{I}_n|\},
\end{equation}
where for each $j\in \{1,2,\ldots,|\mathcal{I}_n|\}$, $\partial^4_{\alpha_n(j)}$ denotes the fourth partial derivative in terms of $x\in \mathbb{R}^n$ with respect to the index $\alpha_n(j)$.
We emphasize that~\eqref{eq:Pzmultione}  does not determine all the elements of $P$ uniquely as functions of the elements of $Q$ and $c$, in some cases we only obtain some linear relations they must fulfil.
Further, note that the coefficients of $P$ depend on $c$ and the coefficients of $Q$,
see Remark~\ref{rem:compuentries} for a concrete example when $n=2$ and $\ell=1$ in~\eqref{eq:modelito}.

We now claim that there are:
\begin{itemize}
\item[(1)] No dependent variables if $\alpha_n$ is of class 1 or class 2.
\item[(2)] Two dependent variables if $\alpha_n$ is of class 3 or class 4.
\item[(3)] Three dependent variables if $\alpha_n$ is of class 5.
\end{itemize}
Indeed, set $z=(z_j)_{j\in\{1,\ldots,m\}}$, where for each $j\in \{1,\ldots,m\}$ the entry $z_j$ is defined in~\eqref{def:zuno}.
Recall that $P=(P_{i,j})_{i,j\in\{1,\ldots,m\}}\in \mathsf{M}_m$.
Then we have
\begin{equation}\label{eq:expansone}
z^*Pz=\sum_{i,j=1}^m z_iz_j P_{i,j}.
\end{equation}

We start with the proof of Item~(1).
Assume that $\alpha_n$ is of class 1.
For a given  $r\in \{1,2,\ldots,n\}$, we
observe that $z_iz_j=x^4_r$ for some $i:=i(r),j:=j(r)\in \{1,\ldots,m\}$, if and only if $z_i=x^2_r$, $z_j=x^2_r$ and $i=j$.
Note that the coefficient of $z_iz_i$ in~\eqref{eq:expansone} is $P_{i,i}$. Hence
applying
~\eqref{eq:Pzmultione} for $\alpha_n$ yields that there are no dependent variables necessary, i.e.,~we can write the corresponding $P_{i,i}$ directly in terms of the elements of $Q$ and $c$.

Now, assume that $\alpha_n$ is of class 2.
Let $r,s\in\{1,2,\ldots,n\}$ with $r<s$.
To obtain a term of the form $x^3_r x_s$ we must have for some $i,j\in \{1,\ldots,m\}$
that
$z_i=x^2_r$, $z_j=x_rx_s$ or
$z_i=x_rx_s$, $z_j=x^2_r$.
Then the coefficients of $z_iz_j$ and $z_jz_i$ in~\eqref{eq:expansone} are given by $P_{i,j}$ and $P_{j,i}$, respectively.
Hence
applying
~\eqref{eq:Pzmultione} for $\alpha_n$ gives
$P_{i,j}+P_{j,i}=2P_{i,j}$ since $P$ is a symmetric matrix. Therefore, there are no
dependent variables necessary.

We continue with the proof of Item~(2).
Assume that $\alpha_n$ is of class 3.
Let $r,s\in\{1,2,\ldots,n\}$ with $r<s$.
To get a term of the form $x^2_rx^2_s$ we must have for some $i,j\in \{1,\ldots,m\}$
that
$z_i=x^2_r, z_j=x^2_s$, ($i\neq j$);
$z_i=x^2_s$, $z_j=x^2_r$, ($i\neq j$); or  $z_i=x_rx_s$, $z_j=x_rx_s$, ($i= j$),
whose coefficients of $z_iz_j$ in~\eqref{eq:expansone} are given by $P_{i,j}$, $P_{j,i}$ and $P_{k,k}$ (for $k=i=j$), respectively. This yields two dependent variables after applying~\eqref{eq:Pzmultione} for $\alpha_n$ since we obtain $P_{i,j}+P_{j,i}+P_{k,k}=2P_{i,j}+P_{k,k}$.

Now, assume that $\alpha_n$ is of class 4.
Let $r,s,t\in\{1,2,\ldots,n\}$ with $r<s<t$.
A term of the form $x^2_rx_sx_t$ can be only obtained for some $i,j\in \{1,\ldots,m\}$
such that
$z_i=x^2_r$, $z_j=x_sx_t$, ($i\neq j$); or
$z_i=x_sx_t$, $z_j=x^2_r$, ($i\neq j$); or $z_i=x_sx_r$, $z_j=x_rx_t$, ($i\neq j$); or $z_i=x_rx_t$, $z_j=x_rx_s$, ($i\neq j$).
Thus
applying
~\eqref{eq:Pzmultione} for $\alpha_n$ gives
$2P_{i,j}+2P_{k,\ell}$ since $P$ is a symmetric matrix. Therefore, there are two
dependent variables. This means that we can only write the sum $2P_{i,j}+2P_{k,\ell}$ in terms of the elements of $Q$ and $c$.

Finally, we  show Item~(3).
Assume that $\alpha_n$ is of class 5.
Let $r,s,t,u\in\{1,2,\ldots,n\}$ with $r<s<t<u$.
A term of the form $x_rx_sx_tx_u$ can be only obtained for some $i,j\in \{1,\ldots,m\}$
such that
$z_i=x_rx_s$, $z_j=x_tx_u$, ($i\neq j$); or
$z_i=x_tx_u$, $z_j=x_rx_s$, ($i\neq j$); or $z_i=x_rx_t$, $z_j=x_sx_u$, ($i\neq j$); or $z_i=x_sx_u$, $z_j=x_rx_t$, ($i\neq j$); or
$z_i=x_rx_u$, $z_j=x_sx_t$, ($i\neq j$); or
$z_i=x_sx_t$, $z_j=x_rx_u$, ($i\neq j$).
Therefore,
applying
~\eqref{eq:Pzmultione} for $\alpha_n$ gives
$2P_{i,j}+2P_{k,\ell}+2P_{p,q}$ since $P$ is a symmetric matrix. Therefore, there are three dependent variables.

As a consequence we have that the number of dependency relations in a problem of size $n $ is given by
\begin{equation}\label{eq:deprela}
\begin{split}
R:&=|\mathcal{C}_3|+|\mathcal{C}_4|+|\mathcal{C}_5|=n\frac{(n-1)}{2}+n\frac{(n-1)(n-2)}{2}+{{n}\choose{4}}\mathds{1}_{\{n\geq 4\}}\\
&=n\frac{(n-1)^2}{2}+{{n}\choose{4}}\mathds{1}_{\{n\geq 4\}}.
\end{split}
\end{equation}

\noindent
\textbf{Step 2: the BMI problem.}
To write inequality~\eqref{eq:Hine} as a BMI we need to find real numbers (variables) $q_1,q_2,\ldots,q_K$ and symmetric matrices
$A_{(i,j)}$, $B_{(i)}, C\in \mathsf{M}_N$ for $i,j=1,\ldots,m$
such that
\begin{equation}\label{eq:defJmatrix}
J:=\sum_{i,j=1}^{m} q_iq_jA_{(i,j)}+\sum_{i=1}^K q_iB_{(i)}+C\succeq O_{N\times N}
\end{equation}
is equivalent to~\eqref{eq:Hine}.
We start by fixing the values  of $K$ and $N$ through
\begin{equation}\label{eq:defN}
\begin{split}
N&:=m+2R+n+1\\
&=
\frac{1}{2}(2n^3-3n^2+5n+2)+2{{n}\choose{4}}\mathbf{1}_{\{n \geq 4\}}
\end{split}
\end{equation}
and
\begin{equation}\label{eq:defK}
\begin{split}
K&:=m+1+2R\\
&=
\frac{1}{2}(2n^3-3n^2+3n+2)+2{{n}\choose{4}}\mathbf{1}_{\{n \geq 4\}}.
\end{split}
\end{equation}
Formulas~\eqref{eq:defN}
and~\eqref{eq:defK} can be found
in~\cite{Bjarkason}. However, we point out that there is a typo in the third equality for the formula for $N$ on p.~30.
We note that $J\in \mathsf{M}_N$.
In the sequel, we assign the variables
$q_1,q_2,\ldots,q_K$ of the BMI to the original variables $c$ and the entries of the matrix $Q$ and the needed entries of the  matrix $P$ (the dependent ones).
Recall that $Q=(Q_{i,j})_{i,j\in \{1,\ldots,n\}}$ and $P=(P_{i,j})_{i,j\in \{1,\ldots,m\}}$ with $m=n(n+1)/2$.
Let
\begin{equation}\label{eq:defqlabel}
\begin{split}
q_{j}&:=Q_{1,j}\qquad \textrm{ for }\quad j\in \{1,\ldots,n\},\\
q_{n+j-1}&:=Q_{2,j}\qquad \textrm{ for }\quad j\in \{2,\ldots,n\},\\
q_{n+(n-1)+j-2}&:=Q_{3,j}\qquad \textrm{ for }\quad j\in \{3,\ldots,n\},\\
&\,\,\vdots\\
q_{n+(n-1)+\cdots+3+j-(n-2)}&:=Q_{n-1,j}\quad \textrm{ for }\quad j\in \{n-1,n\},\\
q_{m}&:=Q_{n,n}.
\end{split}
\end{equation}

For simplicity, we set $q_K:=c$.
In other words, the first $m$-variables are set to the entries of $Q$ (only upper diagonal entries since $Q$ is a symmetric matrix) and the last variable $q_K$ to the variable $c$.
Now, we set the
value for the $(K-m-1)$-variables $(q_j)_{m+1\leq j \le K-1}$.
We stress that some entries of $P$ can be written in terms of the elements of $Q$ and $c$. Hence, for such entries of $P$ we do not need to assign $q$-variables.
The remainder $q$-variables are assignment to the elements of $P$ that posses dependent relations. In other words, the remaining $q$-variables
are assigned to the elements of $P$ that give dependency  relations in~\eqref{eq:Pzmultione} and there are exactly $R$ dependency relations, see~\eqref{eq:deprela}.

We point out that in order to implement an equality of the type $\ast=0$ for a dependent relation, we implement
\begin{equation}\label{eq:impequ}
\ast \geq 0\quad  \textrm{ and } -\ast\geq 0.
\end{equation}
Hence, we need $K$ variables $q_1,\ldots,q_K$, where $K$ is given in~\eqref{eq:defK}.

In the sequel, we show how to implement the constraints $Q\succ O_{n\times n}$, $P\succeq O_{m\times m}$ and $c>0$ into the BMI~\eqref{eq:defJmatrix}.
The semi-definite constraint $P\succeq O_{m\times m}$ is implemented in the principal sub-matrix $J_{1:m,1:m}$,
where for each $\ell_1,\ell_2,r_1,r_2\in
\{1,\ldots,m\}$ with $\ell_1\leq \ell_2$ and $r_1\leq r_2$ we denote the minor matrix of dimension $(\ell_2-\ell_1+1)\times(r_2-r_1+1)$ of $J$ by
$J_{\ell_1:\ell_2,r_1:r_2}$.
 Recall that the BMI problem~\eqref{eq:defJmatrix} is formulated in terms of semi-definite constraints instead of definite constraints.
Hence, we rewrite  $Q\succ O_{n\times n}$ and $c>0$ as $Q-\varepsilon I_n\succeq O_{n\times n}$ and $c-\varepsilon \geq 0$ for some $\varepsilon>0$, respectively.
Now, we implement the constraint $Q-\varepsilon I_n\succeq O_{n\times n}$ in the principal sub-matrix $J_{N-n:N-1,N-n:N-1}$ and the constraint $c-\varepsilon \geq 0$ in $J_{N,N}=J_{N:N,N:N}$. Finally, the remainder of the diagonal elements of $J$ are reserved for the implementation of the
dependency relations  (recalling~\eqref{eq:impequ}) and all other elements of $J$ are set to zero. Visually, the constraints are implemented as follows:
\begin{equation}
\begin{pmatrix}
P & & & & & & &\\
  & \check{\mu}_1-\hat{\mu}_1 & & & & & &\\
  &  & \hat{\mu}_1-\check{\mu}_1 & & & & &\\
  &  & & \ddots & & & &\\
  &  & &   & \check{\mu}_R-\hat{\mu}_R & & &\\
  &  & &   &  & \hat{\mu}_R-\check{\mu}_R & &\\
  &  & &   &  &  & Q-\varepsilon I_n &\\
  &  & &   &  &  &  & c-\varepsilon\\
\end{pmatrix}\succeq 0_{N\times N},
\end{equation}
where $\check{\mu}_r$ and $\hat{\mu}_r$ denote the left-hand side and the right-hand side of a dependency relation, respectively, $r=1,2,\ldots,R$.
In addition,
 the implicit elements that do not appear in the previous matrix representation are set to zero.
We note that the placement of the constraints in $J$ is irrelevant.

In what follows, we construct the symmetric matrices
$A_{(i,j)}$, $B_{(k)}, C\in \mathsf{M}_N$ for $i,j\in \{1,\ldots,m\}$ and $k\in \{1,\ldots,K\}$.
Let $(e_u)_{1\leq u\leq N}$ be the canonical basis of $\mathbb{R}^N$.
For each $u,v\in \{1,\ldots,N\}$ we define the matrix $E_{(u,v)}\in \mathsf{M}_N$ by
\begin{equation}\label{eq:basis}
E_{(u,v)}:=
\begin{cases}
e_ue^*_v & \textrm{ if }\quad u=v,\\
e_ue^*_v+e_ve^*_u & \textrm{ if }\quad u\not=v.
\end{cases}
\end{equation}
We note that $(E_{(i,j)})_{1\leq i\leq j\leq N}$ is a basis for the vector space of the symmetric matrices in $\mathsf{M}_N$.

We start with the construction of the matrix $A_{(i,j)}$ for each $i,j\in \{1,\ldots,m\}$.
Let $i,j\in \{1,\ldots,m\}$ be fixed.
We first observe that the matrix $A_{(i,j)}$ is associated to the variables $q_iq_j$ in~\eqref{eq:defJmatrix}. Hence, the entries of $A_{(i,j)}$ correspond to the coefficients of $q_iq_j$
in~\eqref{eq:Pzmultione}.

On the one hand,
if for some $\ell_1,\ell_2\in \{1,\ldots,m\}$ with $\ell_1\leq \ell_2$ we have the following relation arising from~\eqref{eq:Pzmultione}
\begin{equation}\label{eq:relat1}
P_{\ell_1,\ell_2}=\sum_{i=1}^n\sum_{j=i}^n c_{i,j}q_iq_j+\eta c,
\end{equation}
for some real numbers $c_{i,j}=c_{i,j}(\ell_1,\ell_2)$ for $i,j\in \{1,\ldots,n\}$ with $i\leq j$, and
$\eta=\eta(\ell_1,\ell_2)$.
Then we multiply the coefficient of
$q_i q_j$ with the matrix $E_{(\ell_1,\ell_2)}$, that is, $c_{i,j}E_{(\ell_1,\ell_2)}$.
We recall that the constraint $P\succeq O_{m\times m}$ is implemented in the principal sub-matrix $J_{1:m,1:m}$.

On the other hand,
if for some $\ell_1,\ell_2,r_1,r_2\in \{1,\ldots,m\}$, with $\ell_1\leq \ell_2$ and $r_1\leq r_2$ we have the following relation arising from~\eqref{eq:Pzmultione}
\begin{equation}\label{eq:relat2}
P_{\ell_1,\ell_2}+P_{r_1,r_2}=\sum_{i=1}^n\sum_{j=i}^n c_{i,j}q_iq_j+\eta c,
\end{equation}
for some real numbers $c_{i,j}=c_{i,j}(\ell_1,\ell_2,r_1,r_2)$ for $i,j\in \{1,\ldots,n\}$ with $i\leq j$, and
$\eta=\eta(\ell_1,\ell_2,r_1,r_2)$.
Then we multiply the coefficient of
$q_i q_j$ with the matrix $E_{(u,u)}-E_{(u+1,u+1)}$, where $m+1\leq u\leq m+2R$. That is, $E_{(u,u)}$ and $E_{(u+1,u+1)}$ are the positions in the diagonal where the relations with dependent variables are implemented.
Hence, we have that  the matrix $A_{(i,j)}$ is a linear combination of the coefficients of $q_iq_j$ in the relations~\eqref{eq:relat1}
and~\eqref{eq:relat2} multiplied with the  corresponding basis matrices $E_{(u,v)}$ from~\eqref{eq:basis} to put them in the correct positions.

We continue with the construction of $B_{(i)}$ for $i\in \{1,\ldots,K\}$.
Recall that the matrix $Q$ is symmetric.
Using the same labelling as in~\eqref{eq:defqlabel},
 the matrices $B_{(i)}$ for $i\in \{1,\ldots,m\}$ are set
as $E_{(i,j)}$, where $(i,j)$ is the position of the elements of the upper triangular matrix $Q$ in the BMI matrix $J$. More precisely, 
\begin{equation}\label{eq:defBlabel}
\begin{split}
B_{(j)}&:=E_{(N-n,N-n+j-1)}\qquad\quad \,\,\,\, \textrm{ for }\quad j\in \{1,\ldots,n\},\\
B_{(n+j-1)}&:=E_{(N-n+1,N-n+j-1)}\qquad\,\,\,\,\, \textrm{ for }\quad j\in \{2,\ldots,n\},\\
B_{(n+(n-1)+j-2)}&:=E_{(N-n+2,N-n+1+j-2)}\qquad \textrm{ for }\quad j\in \{3,\ldots,n\},\\
&\,\,\vdots\\
B_{(n+(n-1)+\cdots+3+j-(n-2))}&:=E_{(N-2,N-2+j-(n-1))}\qquad\,\, \textrm{ for }\quad j\in \{n-1,n\},\\
B_{(m)}&:=E_{(N-1,N-1)}.
\end{split}
\end{equation}
The matrix $B_{(K)}$ is set as the linear combination of coefficients of the variable $c$ in the dependent relations multiplied by their corresponding matrices $E_{(\cdot,\cdot)}$, similarly as we discuss in~\eqref{eq:relat1} and~\eqref{eq:relat2} for the construction of matrices $A_{(i,j)}$.
The rest of the matrices $(B_{(i)})_{m+1\leq i\leq K-1}$ are used to implement the dependent relations.

Finally, the matrix $C$ is used to implement $-\varepsilon I_n$ and $-\varepsilon$ in the constrains  $Q-\varepsilon I_n\succeq O_{n\times n}$
and $c-\varepsilon\geq 0$, respectively, i.e.
\[
C:=-\varepsilon \sum_{i=N-n}^{N}E_{(i,i)}.
\]

The proof is complete.
\end{proof}

\section{\textbf{Proof of Theorem~4.2 in~\cite{Sarkar}}}\label{ap:sarkar}\hfill

\noindent
In this section, since there are typos in the proof of Theorem~4.2 in~\cite{Sarkar} that yield misconclusions, following the idea given in~\cite{Sarkar} we include the correct statement (see Theorem~\ref{Cneccond} below) of Theorem~4.2 in~\cite{Sarkar} and provide the proof.

The non-linear deterministic (prey-predator) model  introduced in Section~2 of~\cite{Sarkar} is the solution of the ordinary differential equation in $\mathbb{R}^3$ given by
\begin{equation}\label{eq:nonODE}
\ud\left(
\begin{matrix}
M(t)\\
N(t)\\
Z(t)
\end{matrix}
\right)=F(M(t),N(t),Z(t))\ud t,\quad t\geq 0,
\end{equation}
where the  vector field $F:\mathbb{R}^3\to \mathbb{R}^3$ is given by
\begin{equation}\label{eq:defF}
F(M,N,Z):=
\begin{pmatrix}
q+rM\left(1-\frac{M}{k_1}\right)-\alpha MN\\
\beta NZ-d_1N\\
sZ\left(1-\frac{Z}{k_2}\right)-\beta NZ-d_2Z
\end{pmatrix}
\end{equation}
with the following interpretation and restrictions of the variables and parameters:
\begin{itemize}
\item $M=(M(t))_{t\geq 0}$ is the density of tumor cells, i.e.~$M(t)$ is the density of tumor cells at time $t$,
\item $N=(N(t))_{t\geq 0}$ is the density of hunting predator cells, i.e., $N(t)$ is the density of hunting predator cells at time $t$,
\item $Z=(Z(t))_{t\geq 0}$ is the density of
resting predator cell, i.e., $Z(t)$ is the density of
resting predator cell at time $t$,
\item $r>0$ is the growth rate of tumor cells,
\item $q>0$ is the conversion of normal
cells to malignant ones (fixed input),
\item $k_1>0$ is the maximum carrying or packing capacity of tumor cells,
\item $k_2>0$  is the maximum carrying capacity of resting cells, 
\item $\alpha>0$ is
the rate of predation/destruction of tumor cells by the hunting cells,
\item $\beta$ is the rate of conversion
of resting cell to hunting cell,
\item $d_1>0$ is the natural death rate of hunting cells,
\item $d_2>0$ is the natural death rate of resting cells,
\item $s>0$ is the growth
rate of resting predator cells.
\end{itemize}
In addition, assume that $k_1>k_2$, $s>d_2$ and  $\beta>\frac{s d_1}{k_2(s-d_2)}$.
Note that $\beta>0$.
A straightforward computation yields
that there is a unique positive equilibrium of $F$, here denoted by
$E_3=(M^\dag,N^\dag,Z^\dag)$, i.e. $M^\dag$, $N^\dag$ and $Z^\dag$ are all positive numbers and $F(M^\dag,N^\dag,Z^\dag)=(0,0,0)$.
More precisely,
\begin{equation}\label{eq:MNZ}
\begin{split}
M^\dag&:=\frac{-(\alpha N^\dag-r)+\sqrt{(\alpha N^\dag-r)^2+\frac{4rq}{k_1}}}{2\frac{r}{k_1}}>0,\\
N^\dag&:=\frac{s}{\beta}\left(1-\frac{d_1}{\beta k_2} \right)-\frac{d_2}{\beta}>0\quad \textrm{and}\\
Z^\dag&:=\frac{d_1}{\beta}>0.
\end{split}
\end{equation}
The corresponding stochastic non-linear model introduced in Section~3 of~\cite{Sarkar} is the solution of the  stochastic differential equation given by
\begin{equation}\label{eq:nonSDE}
\ud\left(
\begin{matrix}
M(t)\\
N(t)\\
Z(t)
\end{matrix}
\right)=F(M(t),N(t),Z(t))\ud t+G(M(t),N(t),Z(t))
\ud W(t),
\end{equation}
where the vector field $F:\mathbb{R}^3\to \mathbb{R}^3$ is given in~\eqref{eq:defF}, the volatility diffusivity matrix $G:\mathbb{R}^3\to \mathbb{R}^{3\times 3}$ is defined by
\begin{equation}\label{eq:defG}
G(M,N,Z):=
\begin{pmatrix}
\sigma_1(M-M^\dag) & 0 & 0\\
0 & \sigma_2(N-N^\dag) & 0\\
0 & 0& \sigma_3(Z-Z^\dag)
\end{pmatrix}
\end{equation}
with $M^\dag$, $N^\dag$, and $Z^\dag$ given in~\eqref{eq:MNZ}, $\sigma_1$, $\sigma_2$ and $\sigma_3$ are positive constants, and $W:=(W(t))_{t\geq 0}$
\begin{equation}\label{eq:defW}
W(t):=\begin{pmatrix}
W_1(t)\\
W_2(t)\\
W_3(t)
\end{pmatrix},\quad t\geq 0,
\end{equation}
is a standard Brownian motion in $\mathbb{R}^3$.

The linearization of~\eqref{eq:nonSDE} around $E_3$ is the solution of the  following linear stochastic differential equation (centred at the positive equilibrium $E_3$)
\begin{equation}\label{eq:SDEE3}
\ud\left(
\begin{matrix}
\widetilde{M}(t)-M^\dag\\
\widetilde{N}(t)-N^\dag\\
\widetilde{Z}(t)-Z^\dag
\end{matrix}
\right)=DF(M^\dag,N^\dag,Z^\dag)\begin{pmatrix}
\widetilde{M}(t)-M^\dag\\
\widetilde{N}(t)-N^\dag\\
\widetilde{Z}(t)-Z^\dag
\end{pmatrix}\ud t+G(\widetilde{M}(t),\widetilde{N}(t),\widetilde{Z}(t))
\ud W(t),
\end{equation}
where $DF(M^\dag,N^\dag,Z^\dag)$ is the Jacobian of $F$ at $E_3=(M^\dag,N^\dag,Z^\dag)$. Straightforward computations yield
\begin{equation}\label{eq:defA}
A:=DF(M^\dag,N^\dag,Z^\dag)=
\begin{pmatrix}
\delta_1 & \delta_3 & 0\\
0 & 0 & \frac{\beta}{\alpha}\delta_2\\
0 & -d_1 & -\frac{sd_1}{\beta k_2}
\end{pmatrix}
\end{equation}
with
\begin{equation}\label{eq:defdelta}
\begin{split}
\delta_1:&=-\sqrt{\left[\frac{\alpha s}{\beta}\left(1-\frac{d_1}{\beta k_2}\right)-\frac{\alpha d_2}{\beta}-r\right]^2+\frac{4rq}{k_1}}<0,\\
\delta_2:&=\frac{\alpha s}{\beta}\left(1-\frac{d_1}{\beta k_2}\right)-\frac{\alpha d_2}{\beta}=\alpha N^\dag>0 \quad \textrm{and}\\
\delta_3:&=-\alpha M^\dag<0.\\
\end{split}
\end{equation}
Now, we state conditions on the parameters of the model in order that the null solution of the linear stochastic system~\eqref{eq:SDEE3} is exponentially mean-square stable.
\begin{theorem}
\label{Cneccond}\hfill

\noindent
Assume that $r$, $q$, $k_1$, $k_2$, $\alpha$, $\beta$, $d_1$, $d_2$ and $s$ are positive constants.
In addition, assume that $k_1>k_2$, $s>d_2$ and
$\beta>\frac{s d_1}{k_2(s-d_2)}$.
Recall the definition of 
$\delta_1<0$, $\delta_2>0$ and $\delta_3<0$ given in in~\eqref{eq:defdelta}.
Let  $w_3\in \mathbb{R}$ and $w_4\geq 0$ such that $w_3+w_4>0$, $w_2+w_4>0$, with $w_2$ given by
\begin{equation}\label{eq:defw2}
w_2:=\frac{\alpha}{\beta \delta_2}\left[
\left(d_1-\frac{\beta}{\alpha}\delta_2+\frac{sd_1}{\beta k_2} \right)w_4+d_1w_3
 \right],
\end{equation}
and
\begin{equation}\label{eq:nota1}
\frac{-\frac{\beta}{\alpha}\delta_2 w_4}{w_3+w_4} +\frac{sd_1}{\beta k_2}>0.
\end{equation}
Let $\sigma^2_1\geq 0$ be a constant such that
\begin{equation}\label{eq:n1new}
\sigma^2_1<-2\delta_1
\end{equation}
and we additionally assume that
\begin{equation}\label{eq:nota2}
\frac{2d_1 w_4}{w_2+w_4}+\frac{\delta^2_3}{(w_2+w_4)|2\delta_1+\sigma^2_1|}>0.
\end{equation}
Let $\sigma_2^2\geq 0$ and $\sigma^2_3\geq 0$ be constants satisfying
\begin{equation}\label{eq:n1}
\begin{split}
&\sigma^2_2<\frac{2d_1 w_4}{w_2+w_4}+\frac{\delta^2_3}{(w_2+w_4)|2\delta_1+\sigma^2_1|}\quad \textrm{ and }\quad \\
&\sigma^2_3<\frac{-2\frac{\beta}{\alpha}\delta_2 w_4}{w_3+w_4} +2\frac{sd_1}{\beta k_2}.
\end{split}
\end{equation}
Then the null solution of the linear SDE~\eqref{eq:SDEE3} is exponentially mean-square stable.
In particular, the choice $w_3>0$ and $w_4=0$ yields
$w_3+w_4>0$,~$w_2+w_4>0$,~
~\eqref{eq:nota1} and~\eqref{eq:nota2}.
\end{theorem}

\begin{proof}
For short, we write $U_1(t):=\widetilde{M}(t)-M^\dag$,
$U_2(t):=\widetilde{N}(t)-N^\dag$ and $U_3(t):=\widetilde{Z}(t)-Z^\dag$ for all $t\geq 0$.
In the sequel, we construct a Lyapunov function to verify that the null solution of the linear stochastic system~\eqref{eq:SDEE3} is exponentially mean-square stable.
Let $V:\mathbb{R}^3\to [0,\infty)$ be defined by
\begin{equation}
V(u):=\frac{1}{2}[u^2_1+w_2 u^2_2+w_3u^2_3+w_4(u_2+u_3)^2]\quad \textrm{ with }\quad u=(u_1,u_2,u_3)^*,
\end{equation}
where $w_2$, $w_3$ and $w_4$ are suitable constants that ensure that $V(u)\geq 0$ for all $u\in \mathbb{R}^3$. Such constants are defined later on.
Note that
\begin{equation}\label{eq:defV}
\nabla_u V(u)=(u_1,w_2u_2+w_4(u_2+u_3),w_3u_3+w_4(u_2+u_3)),
\end{equation}
where $\nabla_u V(u)$ denotes the gradient of $V$ at $u$,
and
\begin{equation}\label{eq:defH}
\textsf{Hess}_u (V(u))
=\begin{pmatrix}
1 & 0 &0\\
0 & w_2+w_4 & w_4\\
0 & w_4 & w_3+w_4
\end{pmatrix},
\end{equation}
where $\textsf{Hess}_u (V(u))$ denotes the Hessian matrix of the scalar function $V$ at $u$.
By~\eqref{eq:defA} and~\eqref{eq:defV} we have
\begin{equation}\label{eq:part1}
\begin{split}
\langle Au, (\nabla_u V(u))^*
 \rangle&=
\left\langle \begin{pmatrix}
\delta_1 u_1+\delta_3 u_2\\
\frac{\beta}{\alpha}\delta_2 u_3\\
-d_1 u_2-\frac{sd_1}{\beta k_2}u_3
\end{pmatrix}, \begin{pmatrix}
u_1\\
w_2u_2+w_4(u_2+u_3)\\
w_3u_3+w_4(u_2+u_3)
\end{pmatrix}
\right \rangle\\
&=\delta_1 u^2_1+\delta_3 u_1u_2+\frac{\beta}{\alpha}\delta_2 u_3(w_2u_2+w_4(u_2+u_3))\\
&\qquad+\left(-d_1u_2-\frac{sd_1}{\beta k_2}u_3\right)(w_3u_3+w_4(u_2+u_3))\\
&=\delta_1 u^2_1+\delta_3 u_1u_2+\frac{\beta}{\alpha}\delta_2 u_3w_2u_2+\frac{\beta}{\alpha}\delta_2 u_3w_4(u_2+u_3)\\
&\qquad-\left(d_1u_2+\frac{sd_1}{\beta k_2}u_3\right)w_3u_3
-\left(d_1u_2+\frac{sd_1}{\beta k_2}u_3\right)w_4(u_2+u_3)\\
&=\delta_1 u^2_1+(-d_1w_4)u^2_2+\left(\frac{\beta}{\alpha}\delta_2 w_4-\frac{sd_1 w_3}{\beta k_2}-\frac{sd_1w_4}{\beta k_2}\right)u^2_3
\\
&\qquad +\left(\frac{\delta_3}{2}\right)2u_1u_2+ (0)2u_1u_3\\
&\qquad+\left(
\frac{\beta}{2\alpha}\delta_2 w_2+\frac{\beta}{2\alpha}\delta_2 w_4-\frac{d_1w_3}{2}-\frac{d_1w_4}{2}-\frac{sd_1}{2\beta k_2}w_4\right)2u_2u_3,
\end{split}
\end{equation}
where $\langle\cdot,\cdot \rangle$ denotes the standard inner product in $\mathbb{R}^3$.
By~\eqref{eq:defw2} we have that
\[
w_2=\frac{\alpha}{\beta \delta_2}\left[
\left(d_1-\frac{\beta}{\alpha}\delta_2+\frac{sd_1}{\beta k_2} \right)w_4+d_1w_3 \right],
\]
which can be written as
\begin{equation}\label{eq:sat}
\frac{\beta}{\alpha}\delta_2 w_2+\frac{\beta}{\alpha}\delta_2 w_4-d_1w_3-d_1w_4-\frac{sd_1}{\beta k_2}w_4=0.
\end{equation}
By~\eqref{eq:part1} and~\eqref{eq:sat} we obtain
\begin{equation}\label{eq:drift}
\begin{split}
\langle Au, (\nabla_u V(u))^*
 \rangle
&=\delta_1 u^2_1+(-d_1w_4)u^2_2+\left(\frac{\beta}{\alpha}\delta_2 w_4-\frac{sd_1 w_3}{\beta k_2}-\frac{sd_1w_4}{\beta k_2}\right)u^2_3
\\
&\qquad +\left(\frac{\delta_3}{2}\right)2u_1u_2+ (0)2u_1u_3+(0)2u_2u_3.
\end{split}
\end{equation}
Now, we compute the contribution of the Brownian motion  in the generator (the so-called quadratic variation). By~\eqref{eq:defG} and~\eqref{eq:defH} we obtain
\begin{equation}\label{eq:vola}
\begin{split}
\frac{1}{2}\textrm{Trace}\left((G(u))^*\textsf{Hess}_u (V(u)) G(u)\right)=\frac{1}{2}[\sigma^2_1 u^2_1+(w_2+w_4)\sigma^2_2 u^2_2+(w_3+w_4)\sigma^2_3 u_3^2],
\end{split}
\end{equation}
where $\textrm{Trace}$ denotes the usual trace operator defined in $\mathbb{R}^{3\times 3}$.
By~\eqref{eq:drift} and~\eqref{eq:vola} we have that the generator of~\eqref{eq:SDEE3} is given by
\begin{equation}
\begin{split}
(\mathcal{L}V)(u)&:=\langle Au, (\nabla_u V(u))^*
 \rangle+\frac{1}{2}\textrm{Trace}\left((G(u))^*\textsf{Hess}_u (V(u)) G(u)\right)\\
 &=\delta_1 u^2_1+(-d_1w_4)u^2_2+\left(\frac{\beta}{\alpha}\delta_2 w_4-\frac{sd_1 w_3}{\beta k_2}-\frac{sd_1w_4}{\beta k_2}\right)u^2_3
\\
&\qquad +\left(\frac{\delta_3}{2}\right)2u_1u_2+ (0)2u_1u_3+(0)2u_2u_3\\
&\qquad+\frac{1}{2}[\sigma^2_1 u^2_1+(w_2+w_4)\sigma^2_2 u^2_2+(w_3+w_4)\sigma^2_3 u_3^2]\\
&=\left(\delta_1+\frac{1}{2}\sigma^2_1\right)u^2_1+
\left(-d_1w_4+\frac{1}{2}(w_2+w_4)\sigma^2_2\right)u^2_2\\
&\qquad \left(\frac{\beta}{\alpha}\delta_2 w_4-\frac{sd_1 w_3}{\beta k_2}-\frac{sd_1w_4}{\beta k_2}+\frac{1}{2}(w_3+w_4)\sigma^2_3\right)u^2_3\\
&\qquad+\left(\frac{\delta_3}{2}\right)2u_1u_2+ (0)2u_1u_3+(0)2u_2u_3=-u^*Q u,
\end{split}
\end{equation}
where
$Q=(Q_{i,j})_{i,j\in \{1,2,3\}}\in \mathbb{R}^{3\times 3}$
with
\begin{equation}
\begin{split}
Q_{1,1}:&=-\delta_1-\frac{1}{2}\sigma^2_1,\quad Q_{2,2}:=d_1w_4-\frac{1}{2}(w_2+w_4)\sigma^2_2,\\
Q_{3,3}:&=-\frac{\beta}{\alpha}\delta_2 w_4+\frac{sd_1 w_3}{\beta k_2}+\frac{sd_1w_4}{\beta k_2}-\frac{1}{2}(w_3+w_4)\sigma^2_3, \quad Q_{1,2}:=Q_{2,1}:=-\frac{\delta_3}{2},\\
Q_{1,3}:&=Q_{2,3}:=Q_{3,1}:=Q_{3,2}:=0.
\end{split}
\end{equation}
We stress that $Q$ is a symmetric matrix.
By Sylvester's criterion (see Theorem~7.2.5 p.~439~in~\cite{Horn}) we have that $Q$ is positive definite ($Q$ has positive eigenvalues), if and only if,
\begin{equation}\label{eq:nr}
Q_{1,1}>0,\quad Q_{3,3}>0
\quad \textrm{ and }\quad
Q_{1,1}Q_{2,2}-Q_{2,1}Q_{1,2}>0.
\end{equation}

It is not hard to see that~\eqref{eq:n1new} and~\eqref{eq:n1} are equivalent to~\eqref{eq:nr}.
Finally, applying Theorem~5.11 in~\cite{Khasminskii} we conclude the statement of the theorem.
\end{proof}

\section*{\textbf{Declarations}}\hfill

\noindent
\textbf{Acknowledgments.}
\hfill

\noindent
The authors are greatly indebted to Dr.~Mohsen Kheirandishfard for his advise on solving BMI problems and to Prof.~Waldemar Barrera (Autonomous University of Yucatan, UADY, Mexico)
for providing to us the reference~\cite{Harpe} and pointing out the Table-Tennis Lemma.
G. Barrera would like to express his gratitude to University of Helsinki (Finland), Aalto University (Finland) and the Instituto Superior T\'ecnico (Portugal), for all the facilities used along the realization of this work. He also would like to thank the University of Iceland for the hospitality and support during the research visit in May 2023, where partial work on this paper was undertaken.

\noindent
\textbf{Funding.}
\hfill

\noindent
The research done for this paper was partially supported by the Icelandic Research Fund (Rann\'is) in the project ``Lyapunov Methods and Stochastic Stability'' (152429-051), which is gratefully acknowledged.
The research of G.~Barrera has been supported by the Academy of Finland,
via an Academy project (project No. 339228) and the Finnish Centre of Excellence in Randomness and Structures (decision numbers 346306 and 346308).

\medskip

\noindent
\textbf{Ethical approval.}
\hfill

\noindent
Not applicable.

\medskip

\noindent
\textbf{Competing interests.}
\hfill

\noindent
The authors declare that they have no conflict of interest.

\medskip

\noindent
\textbf{Authors' contributions.}
\hfill

\noindent
All authors have contributed equally to the paper.

\medskip

\noindent
\textbf{Availability of data and materials.}
\hfill

\noindent
Data sharing not applicable to this article as no data-sets were generated or analyzed during the current study.

\bibliographystyle{amsplain}

\end{document}